\documentclass[11pt,reqno,a4]{amsart}

\usepackage{a4}
\usepackage{graphicx}
\usepackage{amsmath,amssymb,natbib,color,bbm,ifthen,enumerate}
\usepackage{hyperref}
\usepackage[latin1]{inputenc}

\newtheorem{theorem}{Theorem}
\newtheorem{lemma}{Lemma}
\newtheorem{definition}{Definition}
\newtheorem{proposition}{Proposition}
\newtheorem{corollary}{Corollary}
\theoremstyle{remark}
\newtheorem{remark}{Remark}

\theoremstyle{definition}

\renewcommand{\hat}{\widehat}

\newenvironment{enum_H}
  {%
  \setlength{\leftmargini}{4em}\begin{enumerate}}
  {\end{enumerate}}

\newenvironment{enum_C}
  {%
  \setlength{\leftmargini}{2em}
  \begin{enumerate}}
  {\end{enumerate}}

\newenvironment{enum_A}
  {%
  \setlength{\leftmargini}{4em}\begin{enumerate}}
  {\end{enumerate}}

  {\end{enumerate}}

  {\end{enumerate}}

\def\rmd{\mathrm{d}}
\def\rme{\mathrm{e}}
\def\rmi{\mathrm{i}}

\def\cl{\mathop{\stackrel{\mathcal{L}}{\longrightarrow}}}
\def\cp{\mathop{\stackrel{\mathbbm{P}}{\longrightarrow}}}

\def\1{\mathbbm{1}}
\def\bS{\mathbf{S}}
\def\bZ{\mathbf{Z}}

\def\bg{\mathbf{g}}
\def\bc{\mathbf{c}}

\def\bY{\mathbf{Y}}
\def\Zset{\mathbb{Z}}

\def\Rset{\mathbb{R}} 
\def\Cset{\mathbb{C}} 
\def\PE{\mathbb{E}} 
\def\PVar{\mathrm{Var}}
\def\PCov{\mathrm{Cov}}

\def\calN{\mathcal{N}}

\def\eg{\textit{e.g.} }

\newcommand{\eqdef}{\ensuremath{\stackrel{\mathrm{def}}{=}}}

\newcommand{\AVvar}[3][]
{
\ifthenelse{\equal{#1}{}}{\mathbf{V}_{#3}(#2)}{\mathbf{V}_{#3}(#2,#1)}}
\newcommand{\AVvarJoint}[3][]
{
\ifthenelse{\equal{#1}{}}{\mathbf{\Lambda}_{#3}(#2)}{\mathbf{\Lambda}_{#3}(#2,#1)}}

\newcommand{\AsympVarWWE}[2][]
{\rho^2(#2)}

\newcommand{\AVvarInv}[3][]
{\ifthenelse{\equal{#1}{}}{\mathbf{V}^{-1}_{#3}(#2)}{\mathbf{V}^{-1}_{#3}(#2,#1)}}
\newcommand{\sigmaasymp}[2][]
{
\ifthenelse{\equal{#1}{}}{\sigma(#2)}{\sigma(#2,#1)}}

\def\vjsymb{\sigma}
\newcommand{\vj}[4][]{%
\ifthenelse{\equal{#1}{}}{\vjsymb^{#4}_{#2}}{\vjsymb^{#4}_{#2}(#3,#1)}}
\newcommand{\stdj}[3][]{%
\ifthenelse{\equal{#1}{}}{\vjsymb_{#2}}{\vjsymb_{#2}(#3,#1)}}

\newcommand{\hvj}[3][]{%
\ifthenelse{\equal{#1}{}}{\hat{\vjsymb}^2_{#2}}{\hat{\vjsymb}^2_{#2}(#1)}}
\newcommand{\hvjLIM}[3][]{%
\ifthenelse{\equal{#1}{}}{Q^{(#2)}_{#3}}{Q^{(#2)}_{#3}(#1)}}

\def\poldecrease{h}

\def\vseq{v}
\def\vfonc{v^\ast}
\def\aseq{a}
\def\afonc{a^\ast}
\def\vfoncsym{w^\ast}
\def\limcons{\mathrm{C}}
\def\vphase{\Phi}
\def\dimV{N}

\def\decim{\gamma}
\def\localset{\mathrm{L}_{0}}

\title[CLT for arrays of decimated linear processes] 
{Central Limit Theorems for arrays of decimated linear processes} 

\author{F. Roueff}
\address{TELECOM ParisTech, INSTITUT Télécom, CNRS LTCI, 46, rue Barrault, 75634 Paris Cédex 13, France.}
\email{roueff@tsi.enst.fr}
\author{M.S. Taqqu}
\address{Department of Mathematics and Statistics, Boston University Boston, MA 02215, USA.}
\email{murad@math.bu.edu}



\date{May 6, 2008}

\thanks{Murad~S.~Taqqu would like to thank l'\'Ecole Normale Sup\'erieure des T\'elecom\-munications in Paris
for their hospitality.  This research was partially supported by the NSF Grants DMS--0505747 and DMS--0706786 at Boston University.}

\begin{document}

\maketitle

\begin{center}
{\it TELECOM ParisTech and Boston University}\\
\end{center}
\renewcommand{\thefootnote}{}
\footnote{\textit{Corresponding author}: F. Roueff,  TELECOM ParisTech, 46, rue Barrault, 75634 Paris Cédex 13, France.}

\begin{abstract}
Linear processes are defined as a discrete-time convolution between a kernel and an infinite sequence of i.i.d. random
variables. We modify this convolution by introducing decimation, that is, by stretching time accordingly. We then establish
central limit theorems for arrays of squares of such decimated processes. These theorems are used to obtain the asymptotic
behavior of estimators of the spectral density at specific frequencies. Another application, treated elsewhere, concerns the
estimation of the long-memory parameter in time-series, using wavelets.
\end{abstract}


\tableofcontents

\section{Introduction}


Consider a linear process, that is, a weakly stationary sequence
$$
\sum_{t\in\Zset}\vseq(k-t)\xi_k,\quad k\in\Zset,\quad\text{where}\quad\sum_{t\in\Zset}\vseq^2(t)<\infty
$$
and $\{\xi_t, \;t\in\Zset\}$ is a centered white noise sequence, that is an uncorrelated sequence with mean zero. 
We shall sometimes make the following additional assumptions on $\{\xi_t, \;t\in\Zset\}$.

\medskip
\noindent\textbf{Assumptions}~A

\begin{enum_A}
\item\label{it:A1-2moments} $\{\xi_t,\,t\in\Zset\}$ is a sequence of independent and identically distributed  real-valued
  random variables such that  $\PE[\xi_0]=0$, $\PE[\xi_0^2]=1$.
\item\label{it:A1} $\{\xi_t,\,t\in\Zset\}$ satisfies~\ref{it:A1-2moments} and $\kappa_4\eqdef\PE[\xi_0^4]-3$ is finite. 
\end{enum_A}

\medskip

\noindent We will allow decimation and consider, moreover, not one but $N$ linear sequences, all using the same $\{\xi_t,
\;t\in\Zset\}$.
\begin{definition}\label{def:decimatedlinearprocess}
An array of $\dimV$--dimensional decimated linear processes is a process admitting the following linear representation, 
\begin{equation} 
\label{eq:Zdef}
Z_{i,j,k} = \sum_{t\in\Zset} \vseq_{i,j}(\decim_jk-t) \, \xi_t  \;,\quad i=1,\dots,\dimV,\,k\in\Zset,\,j\geq0\;,
\end{equation}
where $\{\xi_t,\;t\in\Zset\}$ is a centered weak white noise, $(\decim_j)_{j\geq0}$ is a diverging sequence of positive integers 
and, for all $i=1,\dots,\dimV$ and $j\geq0$,
$\{\vseq_{i,j}(t),\;t\in\Zset\}$ is  real-valued and satisfies $\sum_{t\in\Zset} \vseq_{i,j}^2(t)<\infty$.
\end{definition}

\begin{remark}
$Z_{i,j,k}$ involves three indices. The index $i=1,\dots,\dimV$ is used to define an $\dimV$-variate version,  the index $j$
labels the decimation factor $\decim_j$, and the index $k$
corresponds to time. 
Because of the presence of the factor $\decim_j$ in~(\ref{eq:Zdef}), $Z_{i,j,k}$
is not a usual convolution.
It can be viewed as a decimated convolution of a white noise in the sense that, after convolution, one keeps only
values spaced by $\decim_j$. A typical choice of decimation is $\decim_j=2^j$, $j\geq0$.
\end{remark}
Our goal is to study the asymptotic behavior of the sample mean square of $Z_{i,j,k}$, namely to find
conditions on the kernels $\vseq_{i,j}$, the decimation factor $\decim_j$ and normalization $n_j$, so that 
the normalized vector 
$$
\left\{n_j^{-1/2}\sum_{k=0}^{n_j-1}\left(Z_{i,j,k}^2-\PE\left[Z_{i,j,k}^2\right]\right),\;i=1,\dots,\dimV\right\}
$$
converges to a multivariate normal $\calN(0,\Gamma)$ distribution. We want also to characterize the limiting covariance
matrix $\Gamma$. Thus we are interested in the sum of squares of the $Z_{i,j,k}$.

Such results are useful in estimation. In section~\ref{sec:appl-non-param}, for example, we apply our result to obtain a
central limit theorem for the estimator $\hat{f}_n(0)$ of the spectral density at the origin $f(0)$ of a linear process. This
CLT is compared to~\cite{bhansali:giraitis:kokoszka:2007}, Eq~(3.9), as discussed in Remark~\ref{rem:compareGiraitisEtAl}.
Another, more involved application, which involves wavelets, can be found in~\cite{roueff:taqqu:2008b}.

The paper is structured as follows. In Section~\ref{sec:main-assumptions}, we indicate the main assumptions. The central
limit theorems (Theorems~\ref{thm:CLTlinear} and~\ref{thm:CLTlinearLocale}) for decimated sequences are stated in
Section~\ref{sec:centr-limit-theor}. Section~\ref{sec:appl-non-param} contains an application to the estimation of the value
of spectral density at the origin (Theorem~\ref{thm:appl-spectr-dens}). 
Section~\ref{sec:technical-lemmas} contains technical lemmas. Theorems~\ref{thm:CLTlinear}, ~\ref{thm:CLTlinearLocale}
and~\ref{thm:appl-spectr-dens} are proved in Sections~\ref{sec:proof-theor-clt},~\ref{sec:proof-theor-cltLocal}
and~\ref{sec:proof-theor-appli} respectively.



\section{Main assumptions}
\label{sec:main-assumptions}

Our assumptions will be expressed in terms of the Fourier series of the $\ell^2$ sequences $\{\vseq_{i,j}(t),\;t\in\Zset\}$, namely
\begin{equation}
\label{eq:vfoncdef}
\vfonc_{i,j}(\lambda)=(2\pi)^{-1/2}\,\sum_{t\in\Zset} \vseq_{i,j}(t)\,\rme^{-\rmi\lambda t} \; .
\end{equation}
We suppose that for any $i=1,\dots,\dimV$, as $j\to\infty$, the Fourier series $\vfonc_{i,j}$ concentrates around some
frequency $\lambda_{i,\infty}\in[0,\pi)$. By ''concentrate'', we mean that when adequately normalized, translated
and rescaled around these frequencies, the series  $\vfonc_{i,j}$ converges as $j\to\infty$ to some limit functions $\vfonc_{i,\infty}$, with a uniform
polynomial control (see Eq~(\ref{eq:unfiBoundvfonc}) and~(\ref{eq:Limitvfonc}) below).
Because of the particular structure of the $\decim_j$--decimation in~(\ref{eq:Zdef}), however,
in order to derive the asymptotic behavior for the processes,  
we need to introduce
sequences of frequencies $(\lambda_{i,j})_{j\geq0}$  that satisfy some special conditions 
and converge to $\lambda_{i,\infty}$ for all  $i=1,\dots,\dimV$. We shall first specify the conditions on the  Fourier
series $\vfonc_{i,j}$, the frequencies $\lambda_{i,\infty}$ and  
the limit functions $\vfonc_{i,\infty}$, and then comment on these conditions.

\medskip

\noindent\textbf{Condition}~C

\begin{enum_C}
\item \label{item:FourierConditions}
There exist a $\dimV$--dimensional array of frequencies $(\lambda_{i,j})_{i\in1,\dots,\dimV,j\geq0}$
valued in $[0,\pi)$ such that, for all $i=1,\dots,\dimV$,
\begin{align}
  \label{eq:integerCondition}
  \decim_j\lambda_{i,j} \in 2\pi\Zset_+,\quad\text{for $j$ large enough}\;,
\end{align}
\begin{equation}
  \label{eq:Limitfreq}
  \lambda_{i,j}\to\lambda_{i,\infty}\;,\quad\text{as $j\to\infty$}\;,
\end{equation}
\begin{align}
  \label{eq:CondFreqEqual0orDifferent} 
\text{if}\quad\lambda_{i,\infty}=0 \;,\quad\text{then}\quad \lambda_{i,j}=0,\quad\text{for $j$ large enough}\;.
\end{align}
and,  for all $1\leq i < i'\leq\dimV$, 
\begin{align}
  \label{eq:CondFreqEqualOrDifferent} 
\text{if}\quad\lambda_{i,\infty}=\lambda_{i',\infty}\;,\quad\text{then}\quad \lambda_{i,j}=\lambda_{i',j},\quad\text{for $j$ large enough}\;.
\end{align}
\item Moreover there exist $\delta>1/2$ and a sequence of $[-\pi,\pi)$-valued functions $\vphase_j(\lambda)$ defined on $\lambda\in\Rset$ such that    
\begin{align}
\label{eq:unfiBoundvfonc}
&\sup_{j\geq0} \sup_{\lambda\in[0,\pi)} \decim_j^{-1/2}|\vfonc_{i,j}(\lambda)|(1+\decim_j|\lambda-\lambda_{i,j}|)^\delta < \infty \;, \\
\label{eq:Limitvfonc}
&\lim_{j\to\infty}\decim_j^{-1/2}\vfonc_{i,j}(\decim_j^{-1}\lambda+\lambda_{i,j})\rme^{\rmi\vphase_j(\lambda)} = \vfonc_{i,\infty}(\lambda) \quad\text{for all}\quad
\lambda\in\Rset\;,
\end{align}
\end{enum_C}

\medskip

\noindent The following remarks provide some insight into these conditions.


\begin{remark}
Equations~(\ref{eq:Limitfreq}) and~(\ref{eq:Limitvfonc}) imply that the spectral density
$\lambda\mapsto|\vfonc_{i,j}|^2(\lambda)$ of the \emph{undecimated} stationary process
$$
W_{i,j,k}= \sum_{t\in\Zset} \vseq_{i,j}(k-t) \, \xi_t  \;,\quad k\in\Zset\;,
$$
concentrates, as $j\to\infty$ around the frequency $\lambda_{i,\infty}$. In practical applications of the theorem, the limiting
frequencies $\{\lambda_{i,\infty},\;i\in1,\dots,\dimV\}$ are given. However, one can often easily find sequences $(\decim_j)_{j\geq0}$
and $(\lambda_{i,j})_{j\geq0}$ that satstisfy Conditions~(\ref{eq:integerCondition}) and~(\ref{eq:Limitfreq}) hold. 
In the particular case where the $\lambda_{i,\infty}$ are such
 that 
$q\,\lambda_{i,\infty}\in2\pi\Zset$ for all $i=1,\dots,\dimV$ and some positive integer $q$, one  
may take $\lambda_{i,j}=\lambda_{i,\infty}$ and $\decim_j$ as a multiple of $q$. This happens for instance when the limiting
frequencies are all at the origin, that is,
$\lambda_{1,\infty}=\dots=\lambda_{\dimV,\infty}=0$ and $\decim_j=2^{j}$. 
\end{remark}
\begin{remark}
  The presence of the phase function $\vphase_j$ in~(\ref{eq:Limitvfonc}) offers flexibility and implies  that 
$\decim_j^{-1/2}\vfonc_{i,j}(\decim_j^{-1}\lambda+\lambda_0)$ converges to $\vfonc_{i,\infty}(\lambda)$ \emph{up to a common change of phase}. 
Observe, however, that $\vphase_j$ should not depend on $i$ and thus,  for $\dimV>1$, Condition~(\ref{eq:Limitvfonc}) is not equivalent to
requiring that $\decim_j^{-1/2}|\vfonc_{i,j}(\decim_j^{-1}\lambda+\lambda_{i,j})|$ converges to
$|\vfonc_{i,\infty}(\lambda)|$ for all $i$. 
The presence of the phase $\vphase$ is consistent with the fact that the asymptotic covariance matrix $\Gamma$ defined
in~(\ref{eq:GammaDef}) is invariant through a common phase translation of the functions $\vfonc_{i,\infty}$ for all
$i=1,\dots,\dimV$. 
\end{remark}
\begin{remark}
Condition~(\ref{eq:CondFreqEqualOrDifferent}) states that if two limits $\lambda_{i,\infty}$ and  $\lambda_{i',\infty}$ are
equal, then the $\lambda_{i,j}$ and  $\lambda_{i',j}$ which converge to them must coincinde for large enough
$j$. Condition~(\ref{eq:CondFreqEqual0orDifferent}) has a similar interpretation. 
\end{remark}
\begin{remark}
Conditions~(\ref{eq:CondFreqEqualOrDifferent}) and~(\ref{eq:Limitvfonc}) imply that, for all $1\leq i\leq i'\leq\dimV$ such 
that $\lambda_{i,\infty}=\lambda_{i',\infty}$, 
\begin{equation}
\label{eq:LimitvfoncAltern}
\lim_{j\to\infty}
\decim_j^{-1}[\vfonc_{i,j}\overline{\vfonc_{i',j}}](\decim_j^{-1}\lambda+\lambda_{i,j}) 
= [\vfonc_{i,\infty}\overline{\vfonc_{i',\infty}}](\lambda) \quad\text{for all}\quad
\lambda\in\Rset\;.
\end{equation}
Here $\overline{z}$ denotes the conjugate of the complex $z$.
\end{remark}
\begin{remark}
Since $\vseq_{i,j}(t)$ is real valued, since $\vseq_{i,j}(t)$ is real-valued, we have
\begin{equation}
\label{eq:symVfonc}
\vfonc_{i,j}(-\lambda)=\overline{\vfonc_{i,j}(\lambda)} \; .
\end{equation} 
Thus, Conditions~(\ref{eq:unfiBoundvfonc}) and~(\ref{eq:LimitvfoncAltern}) imply that  
\begin{align}
\label{eq:unfiBoundvfoncSym}
&\sup_{j\geq0} \sup_{\lambda\in(-\pi,\pi)} \decim_j^{-1/2}|\vfonc_{i,j}(\lambda)|(1+\decim_j\left||\lambda|-\lambda_{i,j}\right|)^\delta < \infty \;, \\
\label{eq:LimitvfoncSym}
&
\lim_{j\to\infty}\decim_j^{-1}[\vfonc_{i,j}\overline{\vfonc_{i',j}}](\decim_j^{-1}\lambda-\lambda_{i,j}) 
= [\overline{\vfonc_{i,\infty}}\vfonc_{i',\infty}](-\lambda) 
\quad\text{for all}\quad \lambda\in\Rset\;.
\end{align}
In particular, if $\lambda_{i,\infty}=\lambda_{i',\infty}=0$,
by~(\ref{eq:CondFreqEqual0orDifferent}),~(\ref{eq:LimitvfoncAltern}) and~(\ref{eq:LimitvfoncSym}), we have 
\begin{equation}
  \label{eq:symetryFreqZero}
  [\vfonc_{i,\infty}\overline{\vfonc_{i',\infty}}](\lambda)=
[\overline{\vfonc_{i,\infty}}\vfonc_{i',\infty}](-\lambda) \;.
\end{equation}
\end{remark}
\begin{remark}
Since $(\decim_j)$ is a diverging sequence and $\lambda_{i,j}\to\lambda_{i,\infty}\in[0,\pi)$, for any $\lambda\in\Rset$, for
$j$ large enough, we have $\decim_j^{-1}\lambda+\lambda_{i,j}\in[0,\pi)$. 
Hence Conditions~(\ref{eq:unfiBoundvfonc}) and~(\ref{eq:Limitvfonc}) imply that, for all $i=1,\dots,\dimV$,  
  \begin{equation}
    \label{eq:LimitvfoncBound}
  \sup_{\lambda\in\Rset}  \left|\vfonc_{i,\infty}(\lambda)\right|\; (1+|\lambda|)^{\delta} < \infty \;.
  \end{equation}
\end{remark}
To better understand these assumptions, we start with a result on the asymptotic behavior of the cross-covariance function 
for the array~(\ref{eq:Zdef}). In this proposition, we set, without loss of generality, $\dimV=2$. 

\begin{proposition}\label{prop:Limcov}
Let $\{Z_{i,j,k},\;i=1,2, j\geq0, k\in\Zset\}$ be an array of $2$--dimensional decimated linear processes as defined by~(\ref{eq:Zdef}).
Assume that Condition~\ref{item:FourierConditions} holds for some $\lambda_{i,\infty}\in[0,\pi)$ and 
functions $\vfonc_{i,\infty}$, $i=1,2$, from $\Rset\to\Zset$. Then, for all $k,k'\in\Zset$, as $j\to\infty$,
\begin{equation}
  \label{eq:Limcov-neq1}
  \PCov\left(Z_{1,j,k},Z_{2,j,k'}\right)  \to \limcons_{1,2}\;
\int_{-\infty}^{\infty}\vfoncsym_{1,2}(\lambda)\;\rme^{\rmi\lambda(k'-k)}\,\rmd\lambda\;,
\end{equation}
where, for any $i,i'\in\{1,2\}$,
\begin{equation}
  \label{eq:vfoncsymDef}
  \vfoncsym_{i,i'}(\lambda)=
\frac12\left[\overline{\vfonc_{i,\infty}(-\lambda)}\vfonc_{i',\infty}(-\lambda)
+\vfonc_{i,\infty}(\lambda)\overline{\vfonc_{i',\infty}(\lambda)}\right],\quad\lambda\in\Rset\;,
\end{equation}
and 
\begin{equation}
  \label{eq:LimConsDef}
\limcons_{i,i'}=\begin{cases}
0&\text{ if $\lambda_{i,\infty}\neq\lambda_{i',\infty}$}\\
1&\text{ if $\lambda_{i,\infty}=\lambda_{i',\infty}=0$}\\
2&\text{ if $\lambda_{i,\infty}=\lambda_{i',\infty}>0$}
\end{cases}\;.
\end{equation}
\end{proposition}
\begin{proof}
Using~(\ref{eq:Zdef}) and Parseval's theorem, we have
\begin{align}
\label{eq:covZ0Z1sum}
\PCov\left(Z_{1,j,k},Z_{2,j,k'}\right)
=\sum_{t\in\Zset}\vseq_{1,j}(\decim_jk-t)\vseq_{2,j}(\decim_jk'-t)\\
\label{eq:covZ0Z1}
=\int_{-\pi}^{\pi}[\vfonc_{1,j}\overline{\vfonc_{2,j}}](\lambda)\,\rme^{\rmi\decim_j\lambda (k'-k)}\,\rmd\lambda\;.
\end{align}

We now consider separately the three cases $\lambda_{1,\infty}\neq\lambda_{2,\infty}$, $\lambda_{1,\infty}=\lambda_{2,\infty}>0$ and
$ \lambda_{1,\infty}=\lambda_{2,\infty}=0$.
\begin{enumerate}[1)]
\item Suppose $\lambda_{1,\infty}\neq\lambda_{2,\infty}$. Then by~(\ref{eq:unfiBoundvfonc}), there is a constant $C>0$ such that
\begin{align}\nonumber
\left|\PCov\left(Z_{1,j,k},Z_{2,j,k'}\right)\right|&\leq C\decim_j\int_0^\pi
(1+\decim_j|\lambda-\lambda_{1,j}|)^{-\delta}(1+\decim_j|\lambda-\lambda_{2,j}|)^{-\delta}\,\rmd\lambda\\
\label{eq:limvInJdifferentFreq}
&\to 0 \quad\text{as $j\to\infty$}\;,
\end{align}
since $\decim_j\to\infty$, $\delta>1/2$ and $|\lambda_{1,j}-\lambda_{2,j}|$ has a positive limit.
\item Suppose $\lambda_{1,\infty}=\lambda_{2,\infty}>0$. Setting $\lambda=\decim_j^{-1}\xi+\lambda_{1,j}$
and using~(\ref{eq:integerCondition}), we have
\begin{align}\nonumber
\int_{0}^{\pi}[\vfonc_{1,j}\overline{\vfonc_{2,j}}](\lambda)\,\rme^{\rmi\decim_j\lambda (k'-k)}\,\rmd\lambda
&=\int_{-\decim_j\lambda_{1,j}}^{\decim_j(\pi-\lambda_{1,j})}\decim_j^{-1}
          [\vfonc_{1,j}\overline{\vfonc_{2,j}}](\decim_j^{-1}\xi+\lambda_{1,j})\,\rme^{\rmi\xi (k'-k)}\,\rmd\xi \\
\label{eq:limvInJ}
& \to\int_{-\infty}^{\infty}[\vfonc_{1,\infty}\overline{\vfonc_{2,\infty}}](\xi)\,\rme^{\rmi\xi (k'-k)}\,\rmd\xi  
\quad\text{as $j\to\infty$,} 
\end{align}
where the limit follows from
Conditions~(\ref{eq:CondFreqEqualOrDifferent}),~(\ref{eq:Limitfreq}),~(\ref{eq:unfiBoundvfonc}),~(\ref{eq:LimitvfoncAltern})
and dominated convergence. Similarly we have 
$$
\int_{-\pi}^0[\vfonc_{1,j}\overline{\vfonc_{2,j}}](\lambda)\,\rme^{\rmi\decim_j\lambda (k'-k)}\,\rmd\lambda 
\to\int_{-\infty}^{\infty}[\overline{\vfonc_{1,\infty}}\vfonc_{2,\infty}](-\xi)\,\rme^{\rmi\xi (k'-k)}\,\rmd\xi  
\quad\text{as $j\to\infty$,} 
$$
by using~(\ref{eq:LimitvfoncSym}) instead of~(\ref{eq:LimitvfoncAltern}). The last display,~(\ref{eq:covZ0Z1})
and~(\ref{eq:limvInJ}) yield
\begin{align}\label{eq:limvInJPosSameFreq}
\PCov\left(Z_{1,j,k},Z_{2,j,k'}\right)\to 2\;\int_{-\infty}^{\infty}
\vfoncsym_{1,2}(\xi)\,\rme^{\rmi\xi (k'-k)}\,\rmd\xi  
 \quad\text{as $j\to\infty$}\;.
\end{align}
\item Suppose finally $\lambda_{1,\infty}=\lambda_{2,\infty}=0$. Setting $\lambda=\decim_j^{-1}\xi$ gives
\begin{align*}
\int_{-\pi}^{\pi}[\vfonc_{1,j}\overline{\vfonc_{2,j}}](\lambda)\,\rme^{\rmi\decim_j\lambda (k'-k)}\,\rmd\lambda
&=
\int_{-\decim_j\pi}^{\decim_j\pi}\decim_j^{-1}
          [\vfonc_{1,j}\overline{\vfonc_{2,j}}](\decim_j^{-1}\xi)\,\rme^{\rmi\xi (k'-k)}\,\rmd\xi \\
&\to\int_{-\infty}^{\infty}[\vfonc_{1,\infty}\overline{\vfonc_{2,\infty}}](\xi)\,\rme^{\rmi\xi (k'-k)}\,\rmd\xi  
\end{align*}
by using Conditions~(\ref{eq:CondFreqEqual0orDifferent}),~(\ref{eq:unfiBoundvfonc}),~(\ref{eq:LimitvfoncAltern}) and
dominated convergence. The last display,~(\ref{eq:limvInJdifferentFreq}) 
and~(\ref{eq:limvInJ}) yield~(\ref{eq:Limcov-neq1}).
\end{enumerate}
\end{proof}

\section{Main results}
\label{sec:centr-limit-theor}

We let $\cl$ denote the convergence in law. 
Our first result provides the asymptotic behavior of the sample mean square of an array of a decimated linear sequence under
a global assumption on the 
behavior of the spectral density (the bound~(\ref{eq:unfiBoundvfonc})). A local version of this assumption is considered in
Theorem~\ref{thm:CLTlinearLocale}. 

\begin{theorem}\label{thm:CLTlinear}
Let $\{Z_{i,j,k},\;i=1,2, j\geq0, k\in\Zset\}$ be an array of $\dimV$--dimensional decimated linear processes as defined by~(\ref{eq:Zdef}).
Assume~\ref{it:A1} and that $\decim_j$ is even for $j$ large enough. 
For each  $i=1,\dots,\dimV$, we let $\lambda_{i,\infty}$ denote a frequency in $[0,\pi)$ and 
$\vfonc_{i,\infty}$ a continuous $\Rset\to\Zset$ function such that Condition~\ref{item:FourierConditions} holds.
Then, for any diverging sequence $(n_j)$,
\begin{equation}
\label{eq:CenteredZ}
n_j^{-1/2} \sum_{k=0}^{n_j-1} 
\left[
\begin{array}{c}
Z_{1,j,k}^2 -\PE[Z_{1,j,k}^2]\\
\vdots\\
Z_{\dimV,j,k}^2 -\PE[Z_{\dimV,j,k}^2]
\end{array}
\right]\cl\calN(0,\Gamma) \; ,
\end{equation}
where $\Gamma$ is the covariance matrix defined by
\begin{equation}
  \label{eq:GammaDef}
  \Gamma_{i,i'}= 
4\pi \,\limcons_{i,i'}\,\int_{-\pi}^\pi \left|\sum_{p\in\Zset}\vfoncsym_{i,i'}(\lambda+2p\pi)\right|^2 \, \rmd\lambda\;,\quad 1\leq i, i'\leq \dimV\; , 
\end{equation}
where $\limcons_{i,i'}$ and $\vfoncsym_{i,i'}$ are defined in~(\ref{eq:LimConsDef}) and~(\ref{eq:vfoncsymDef}).
\end{theorem}

\begin{remark}
 From~(\ref{eq:LimitvfoncBound}), it follows that the doubly infinite sum in~(\ref{eq:GammaDef}) is well defined and bounded on $\lambda\in\Rset$ and hence $\Gamma$ is well defined.
\end{remark}

\begin{remark}
The number of time indices $k$ appearing in the centered sum in~(\ref{eq:CenteredZ}) is $n_j$ and asymptotic normalization occurs as $j$ and $n_j$ tend to $\infty$. 
\end{remark}

\begin{remark}
The presence of the factor $\decim_j$ in~(\ref{eq:Zdef}), and hence of decimation, is essential for the Central Limit Theorem to
hold in this generality because it ensures that the dependence of the $Z_{i,j,k}$'s decreases sufficiently fast as
$j\to\infty$.  Decimation of this type is typically encountered in settings involving wavelets, or more generally filter
banks, see \cite{mallat:1998}. 
\end{remark}
\begin{remark}\label{rem:CenteringTermindep_of_j}
In applications, the expectations in~(\ref{eq:CenteredZ}), which depend on $j$, will be approximated by quantities that
are independent of $j$. To see why this is possible, observe that, applying Relation~(\ref{eq:Limcov-neq1}) in Proposition~\ref{prop:Limcov}
with $k=k'=0$ and $i=i'=1,\dots,\dimV$, we get 
\begin{equation}
    \label{eq:varianceZlim}
\lim_{j\to\infty}\PE\left[Z_{i,j,0}^2\right]=\int_{-\infty}^{\infty}\left|\vfonc_{i,\infty}(\lambda)\right|^2\,\rmd\lambda <
\infty \; .
\end{equation}
Thus, when the convergence rate to this limit is fast enough, the expectations in~(\ref{eq:CenteredZ}) can be replaced by 
$\int_{-\infty}^{\infty}\left|\vfonc_{i,\infty}(\lambda)\right|^2\,\rmd\lambda$, $i=1,\dots,\dimV$, which does not depend on $j$.
\end{remark}  

We have assumed in~(\ref{eq:unfiBoundvfonc}) 
 a bound for $\vfonc_{i,j}(\lambda)$ for $\lambda\in(-\pi,\pi)$. This bound implies
that the spectral density of the process $Z_{i,j,\centerdot}$ defined in~(\ref{eq:Zdef}) is bounded on $(-\pi,\pi)$. We shall
weaken this assumption by only assuming a local bound around the frequency $\lambda_{i,j}$ as follows.
\begin{theorem}\label{thm:CLTlinearLocale}
Assume that all the conditions of Theorem~\ref{thm:CLTlinear} hold except that~(\ref{eq:unfiBoundvfonc}) is replaced by
\begin{equation}
\label{eq:unfiBoundvfoncLocale}
\sup_{j\geq0} \sup_{|\lambda-\lambda_{i,\infty}|\leq\varepsilon}
\decim_j^{-1/2}|\vfonc_{i,j}(\lambda)|(1+\decim_j\left|\lambda-\lambda_{i,j}\right|)^\delta < \infty \;,  
\end{equation}
where  $\varepsilon>0$ is arbitrary small. Suppose in addition that
\begin{equation}
\label{eq:BoundvfoncLocaleNonZero}
n_j^{1/2}\;\int_{0}^\pi\1(|\lambda-\lambda_{i,\infty}|>\varepsilon)\;
|\vfonc_{i,j}(\lambda)|^2\;\rmd\lambda \to0\quad\text{as}\quad j \to\infty \;.
\end{equation}
Then the conclusion of Theorem~\ref{thm:CLTlinear}, that is, the CLT~(\ref{eq:CenteredZ}), still holds.
\end{theorem}

\begin{remark}
Since~(\ref{eq:unfiBoundvfonc}) is replaced by the local condition 
\eqref{eq:BoundvfoncLocaleNonZero},  we impose the additional condition~(\ref{eq:BoundvfoncLocaleNonZero}) on the growth of
$n_j$. This condition
does not appear in the conditions of Theorem~\ref{thm:CLTlinear}, where it was only required that $n_j\to\infty$.
\end{remark}

\section{Application to spectral density estimation}\label{sec:appl-non-param}

Let $\{X(u),\;u\in\Zset\}$ be a standard linear process,
\begin{equation}
\label{eq:Xdef}
X_u=\sum_{t\in\Zset}\aseq(u-t)\,\xi_t\;,
\end{equation}
where $\{\xi_t,\;t\in\Zset\}$ is a centered weak white noise with unit variance and $\{\aseq(t),\,t\in\Zset\}$ is real-valued
sequence such that $\sum_k\aseq_k^2<\infty$ with Fourier series
\begin{equation}
\label{eq:afoncdef}
\afonc(\lambda)=(2\pi)^{-1/2}\,\sum_{t\in\Zset} \aseq(t)\,\rme^{-\rmi\lambda t} \; .
\end{equation}
Then $\{X_k,\;k\in\Zset\}$ admits the following spectral density
$$
f(\lambda)=\left|\afonc(\lambda)\right|^2,\quad\lambda\in(-\pi,\pi)\;.
$$
For simplicity, as in Section~3 of \cite{bhansali:giraitis:kokoszka:2007}, we consider the problem of estimating $f(0)$ from
observations $X_1,\dots,X_n$. 

Let us denote by $W$ a bounded $\Rset\to\Rset$ function with compact support and by $\hat{W}$ its Fourier transform,
$$
\hat{W}(\xi)=\int_{-\infty}^\infty W(t) \rme^{-\rmi \xi t} \, \rmd t \; .
$$

Let $(\decim_j)$ be any diverging sequence of even integers. 

We let $\dimV=1$, and $\lambda_{1,j} = \lambda_ {1,\infty}=0$ for all $j\geq0$, which
 yields
(\ref{eq:integerCondition}),~(\ref{eq:Limitfreq}),~(\ref{eq:CondFreqEqual0orDifferent})
and~(\ref{eq:CondFreqEqualOrDifferent}) in Condition~\ref{item:FourierConditions}.

Define
\begin{equation}
  \label{eq:ZestSpecDef}
Z_{1,j,k}=\decim_j^{-1/2}\sum_{u\in\Zset}W(k - \decim_j^{-1} u )\,X_u \; .
\end{equation}

We assume that
\begin{enum_H}
\item\label{item:betaAssump} As $\lambda\to0$,
\begin{equation}
  \label{eq:fSmooth}
  f(\lambda)=f(0)+O(|\lambda|^2)\;.
\end{equation}
\item\label{item:appl-spectr-dens} the support of $W$ is included in $[-1,0]$,
  $\sup_{\xi\in\Rset}|\hat{W}(\xi)|(1+|\xi|)^{\beta}<\infty$ with $\beta>1$ and 
\begin{equation}
    \label{eq:WhatNorm}
  \int_{-\infty}^\infty\left|\hat{W}(\lambda)\right|^2\;\rmd\lambda=1
\end{equation}
\end{enum_H}

Assumptions~\ref{item:betaAssump} and~\ref{item:appl-spectr-dens} are related to the standard bias control of kernel
estimates of the spectral density (see Lemma~\ref{lem:BijApproxBound} below).

Define
\begin{equation}
  \label{eq:njDefEstSpec}
  n_j=[\decim_j^{-1}(n+1)] \; .
\end{equation}

For all $k=0,\ldots,n_j-1$, since $W(k - \decim_j^{-1} u )$ vanishes for $u\leq0$ and $u\geq n+1$, we have
$$
Z_{1,j,k}=\decim_j^{-1/2}\sum_{u=1}^nW(k - \decim_j^{-1} u )\,X_u \; .
$$
In other words, $\{Z_{1,j,k},\,k=0,\ldots,n_j-1\}$ can be computed from the $n$ observations $X_1,\dots,X_n$.
Thus
$$
\hat{f}_n(0)=n_j^{-1}\sum_{k=0}^{n_j-1}Z_{1,j,k}^2
$$
can be used as an estimator of $f(0)$.
The following theorem provides a central limit result for $\hat{f}_n(0)$.

\begin{theorem}\label{thm:appl-spectr-dens}
Assume~\ref{item:betaAssump} and~\ref{item:appl-spectr-dens} with $\beta>2$. Let $(\decim_n)$ be a diverging sequence of
even integers such that $\decim_n^{-1}n\to\infty$ 
Then, as $n\to\infty$,
\begin{align}
  \label{eq:limitEsp}
\PE\left[\hat{f}_n(0)\right]  =
  \PE\left[Z_{1,n,0}^2\right]  =f(0)+O(\decim_n^{-2}) \; .
\end{align}
If moreover~\ref{it:A1} in Section~\ref{sec:centr-limit-theor} holds and
\begin{equation}
  \label{eq:njCondEstSpec}
  n^{1/2}\decim_n^{1/2-2\beta} \to 0 \;,
\end{equation}
then
\begin{align}\label{eq:nCLTCondEstSpec} 
 (\decim_n^{-1}n)^{1/2} \{\hat{f}_n(0)-\PE\left[Z_{1,n,0}^2\right]\}\cl  \calN(0,\sigma^2)\;,
\end{align}
where
$$
\sigma^2 =2\pi\;f(0)^2\;\int_{-\pi}^\pi\left(\sum_{p\in\Zset}\left|\hat{W}(\lambda+2p\pi)\right|^2\right)^2\;\rmd\lambda \; .
$$
\end{theorem}

\begin{remark}\label{rem:compareGiraitisEtAl}
Our CLT~(\ref{eq:nCLTCondEstSpec}) can be compared with  \cite[Eq.~3.9]{bhansali:giraitis:kokoszka:2007}, although the estimators are
different since ours involve a decimation and the one in~\cite{bhansali:giraitis:kokoszka:2007} is expresses as a weighted
integral of the standard periodogram.  In~(\ref{eq:nCLTCondEstSpec}), our $\gamma_n$ has a role similar to the $q=q_n$ for their estimator.
Our bias estimate~(\ref{eq:limitEsp}) has a faster decrease than the corresponding one $O(q^{-1})$
in~\cite{bhansali:giraitis:kokoszka:2007}, see the last display in their Section 3.
Our conditions also differ from those of~\cite{bhansali:giraitis:kokoszka:2007}. Our conditions on the weight sequence $\aseq(t)$ is
much more general, since we assume  a polynomial decrease neither of this sequence nor of $\aseq(t)-\aseq(t+1)$, as assumed
for the corresponding (causal) sequence $(\psi_j)$ in Assumption~2.1 of~\cite{bhansali:giraitis:kokoszka:2007}.
Standard results on spectral estimation (see e.g. \cite[Theorem~9.4.1]{anderson:1994}) usually assume the even stronger
condition $\sum_t|\aseq(t)|<\infty$.
On the other hand we do assume that the noise sequence $\{\xi_t\}$ has fourth finite moment 
(see~\ref{it:A1}) while only a finite $2+\beta$ moment (with $\beta>0$ arbitrary small) is assumed
in~\cite{bhansali:giraitis:kokoszka:2007}. It is an open question whether similar moment condition can be used for our estimator
$\hat{f}_n(0)$.  
\end{remark}

\section{Technical lemmas}\label{sec:technical-lemmas}

The following lemma will be used several times.

\begin{lemma}\label{lem:finiteFolding}
Let $g$ be a $(2\pi)$-periodic locally integrable function. Then for all positive integer $\decim$, the function defined by 
$$
g_\decim(\lambda)=\sum_{p=0}^{\decim-1} g(\decim^{-1}(\lambda+2p\pi))
$$
is $(2\pi)$-periodic. Moreover, one has
\begin{equation}
  \label{eq:JfoldingInt}
  \int_{-\pi}^\pi g(\lambda)\;\rmd\lambda=\decim^{-1}\;\int_{-\pi}^\pi g_\decim(\lambda) \;\rmd\lambda \; .
\end{equation}
\end{lemma}
\begin{proof}
  Observe that, for all $\lambda\in\Rset$,
  \begin{align*}
  g_\decim(\lambda+2\pi)&=\sum_{p=0}^{\decim-1} g\left(\decim^{-1}\left(\lambda+\left(p+1\right)2\pi\right)\right)\\
  &=\sum_{p=1}^{\decim-1}
  g\left(\decim^{-1}\left(\lambda+2p\pi\right)\right)+g\left(\decim^{-1}\left(\lambda+\decim2\pi\right)\right)=g_\decim(\lambda)\,,
\end{align*}
  since $g(\decim^{-1}(\lambda+\decim2\pi))=g(\decim^{-1}\lambda)$ by the $(2\pi)$-periodicity of $g$. Hence  $g_\decim(\lambda)$ is $(2\pi)$-periodic.

  With a change of variable, one gets
  $$
  \decim^{-1}\int_{0}^{2\pi} g_\decim(\lambda) \;\rmd\lambda=\sum_{p=0}^{\decim-1} \int_{0}^{2\pi\decim^{-1}} g(\xi+2p\pi\decim^{-1})\;\rmd\xi
  =\int_{0}^{2\pi}g(\xi)\;\rmd\xi \; .
  $$
  Relation~(\ref{eq:JfoldingInt}) follows by $(2\pi)$-periodicity of the integrands. 
\end{proof}
The next lemma relates the rates of decrease of two functions with the rate of decrease of their convolution.
\begin{lemma}\label{lem:polynomDecrease}
Let $\delta>0$. For all $T>0$ and $t_0\in[0,T/2)$, we let $\poldecrease_{T,t_0}(t)$, $t\in\Rset$ be the even and $T$-periodic 
function such that   
$$
\poldecrease_{T,t_0}(t) = (1+|t-t_0|)^{-\delta} \quad\text{for all}\quad t\in [0,T/2] \; .
$$
Let $g$ be an integrable non-negative function on $\Rset$ such that
\begin{equation}\label{eq:g-Decrease-assump}
g(t)\leq c_0 |t|^{-\delta-1}\quad\text{for $|t|\geq1$.}
\end{equation}
Then there exists a positive constant $c$, depending only on $\delta$, $\|g\|_1=\int_{-\infty}^\infty g(s)\;ds$ and $c_0$ such 
that, for all $T>0$, $t_0\in[0,T/2)$, and $t\in[0,T/2]$,
\begin{equation}\label{eq:polynDecrease}
g\ast\poldecrease_{T,t_0}(t)=\int_{-\infty}^\infty g(t-u) \poldecrease_{T,t_0}(u) \,du \leq c (1+|t-t_0|)^{-\delta} \; .
\end{equation}
\end{lemma}
\begin{proof}
Let $t_0\in[0,T/2)$.
We shall use the bound, valid for all $t\in\Rset$, 
\begin{equation}
  \label{eq:convolDecreaseSmall-t}
  g\ast\poldecrease_{T,t_0}(t)\leq \|g\|_1 \;.
\end{equation}
This bound yields~(\ref{eq:polynDecrease}) only for $t$ close enough to $t_0$. 
We shall derive a different bound valid only for $t\in[0,T/2]$ with $|t-t_0|\geq2$, namely
\begin{equation}
  \label{eq:convolDecreaseLarge-t}
g\ast\poldecrease_{T,t_0}(t)\leq 2^{1-\delta}\;c_0\;\delta^{-1}\;|t-t_0|^{-\delta}  + \|g\|_1 \; (1+|t-t_0|/2)^{-\delta}\;.
\end{equation}
Applying~(\ref{eq:convolDecreaseSmall-t}) for $|t-t_0|\leq2$ 
and~(\ref{eq:convolDecreaseLarge-t}) for $|t-t_0|\geq2$ yields~(\ref{eq:polynDecrease}).

Hence it only remains to establish~(\ref{eq:convolDecreaseLarge-t}) for $t\in[0,T/2]$ with $|t-t_0|\geq2$. We shall suppose
that $t\in[t_0+2,T/2]$ (the case $t\in[0,t_0-2]$ is obtained similarly).    
Let $u$ such that $|t-u|\leq |t-t_0|/2$. Then  
$t-(t-t_0)/2\leq u\leq t+(t-t_0)/2$ and thus, using that $t\leq T/2$ implies $t+t/2\leq T-t/2$, we get  $(t+t_0)/2\leq u\leq
T-(t+t_0)/2$. Observe that the middle point between $(t+t_0)/2$ and $T-(t+t_0)/2$ is $T/2$. Since
$\poldecrease_{T,t_0}(u)$ is decreasing on $[t_0,T/2]$, and symmetric around $T/2$,
we get $\poldecrease_{T,t_0}(u)\leq\poldecrease_{T,t_0}((t+t_0)/2)=(1+|t-t_0|/2)^{-\delta}$.
Hence we may bound $\poldecrease_{T,t_0}(u)$ by 1 for $|t-u|>|t-t_0|/2$ and by $(1+|t-t_0|/2)^{-\delta}$ otherwise, which gives     
$$
\int_{-\infty}^\infty g(t-u) \poldecrease_{T,t_0}(u) \,du \leq \int_{|s|>|t-t_0|/2} g(s) \;ds + \|g\|_1 \;
(1+|t-t_0|/2)^{-\delta}\;.   
$$
Since $|t-t_0|\geq2$ we may apply the bound~(\ref{eq:g-Decrease-assump}) in the integral of the RHS of the previous
display. Hence we get~(\ref{eq:convolDecreaseLarge-t}), which concludes the proof.
\end{proof}

The following lemma is used, in particular, to bound $\widehat{f}_j$ in the proof of Theorem~\ref{thm:CLTlinearLocale}.
It will be used again in the proof of Lemma~\ref{lem:varlimitCLTlinear} below. 
Applying it, one can bound $g_{j,\decim_j}(\lambda)$ independently of $j$ and $\lambda$, where $g_{j,\decim_j}$ is 
defined as in Lemma~\ref{lem:finiteFolding} with $g$ replaced by $g_j$, and the sequence $g_j$ satisfies a uniform bound of  
the form~(\ref{eq:unfiBoundvfonc}), namely
$$
\sup_{j\geq0}\sup_{\lambda\in[-\pi,\pi)} |g_j(\lambda)|(1+\decim\left| |\lambda| - \lambda_{j}\right|)^\delta <\infty\; ,
$$ 
with $\lambda_j\to\lambda_\infty\in[0,\pi)$ as $j\to\infty$.

\begin{lemma}\label{lem:DomConvArgumenttt}
Let $\delta>1/2$. Then
\begin{equation}\label{eq:supsupgj}
\sup_{t\in\Rset}\sup_{t'\in\Rset}\sum_{p\in\Zset}(1+\left||t+2p\pi|-t'\right|)^{-2\delta} < \infty \;.  
\end{equation}
Moreover, as $u\to\infty$,
\begin{equation}\label{eq:supsupgjLim}
\sup_{t\in\Rset}\sup_{t'\in\Rset}\sum_{p\in\Zset}(1+\left||t+2p\pi|-t'\right|)^{-\delta}(1+\left||t+2p\pi|-t'-u\right|)^{-\delta}
\to0 \;.   
\end{equation}
\end{lemma}
\begin{proof}
Let $S(t,t')=\sum_{p\in\Zset}(3\pi+\left||t+2p\pi|-t'\right|)^{-2\delta}$. Since, for any $t'$, $t\mapsto S(t,t')$ is
$(2\pi)$--periodic we have 
\begin{equation}\label{eq:2piperiodicInt}
\sup_{t\in\Rset}\sup_{t'\in\Rset}S(t,t')=\sup_{t'\in\Rset}\sup_{t:|t-t'|\leq\pi}S(t,t') \;. 
\end{equation}
Suppose that $t,t'\in\Rset$ are such that $|t-t'|\leq\pi$. 
Then for any $a,b\in\Rset$, we have, if  $t+a\geq0$,
\begin{equation}\label{eq:ttprimLeqPi1}
\left||t+a|-t'-b\right|\geq |a-b|-\pi\;,
\end{equation}
and,  if  $t+a\leq0$,
\begin{equation}\label{eq:ttprimLeqPi2}
\left||t+a|-t'-b\right| = \left| t-t' + 2t' + a+b\right|\geq |2t'+a+b| - \pi \;.
\end{equation}
Adding $3\pi$ to each of the last two displays with $a=2p\pi$ and $b=0$, we get that, for all $|t-t'|\leq\pi$ and $p\in\Zset$,   
\begin{equation}\label{eq:ttprimLeqPi1et2}
(3\pi+\left||t+2p\pi|-t'\right|)^{-2\delta} \leq  (2|p|\pi+2\pi)^{-2\delta} +  (|2t'+2p\pi| + 2\pi )^{-2\delta} \; .
\end{equation}
Since $\sum_{p\in\Zset}(2|p|\pi+\pi)^{-2\delta}<\infty$ and $\sup_{t'\in\Rset}\sum_{p\in\Zset} (|2t'+2p\pi| +
\pi)^{-2\delta}<\infty$, Relation~(\ref{eq:2piperiodicInt}) gives that $\sup_{t\in\Rset}\sup_{t'\in\Rset}S(t,t')<\infty$ and~(\ref{eq:supsupgj}) follows. 

We now prove~(\ref{eq:supsupgjLim}). Let 
$$
S(t,t',u)=\sum_{p\in\Zset}(3\pi+\left||t+2p\pi|-t'\right|)^{-\delta}(3\pi+\left||t+2p\pi|-t'-u\right|)^{-\delta}\;.
$$
As above, we have
\begin{equation}\label{eq:2piperiodicInt2}
\sup_{t\in\Rset}\sup_{t'\in\Rset}S_j(t,t')=\sup_{t'\in\Rset}\sup_{t:|t-t'|\leq\pi}S_j(t,t') \;. 
\end{equation}
Suppose that $t,t'\in\Rset$ are such that $|t-t'|\leq\pi$. 
Adding $3\pi$ to~(\ref{eq:ttprimLeqPi1}) and~(\ref{eq:ttprimLeqPi2}) 
with $a=2p\pi$ and $b=u$, we have
$$
(3\pi+\left||t+2p\pi|-t'-u\right|)^{-\delta}
\leq  (|2p\pi-u|+2\pi)^{-\delta} +  (|2t'+2p\pi+u| + 2\pi )^{-\delta} \; .
$$
Using~(\ref{eq:2piperiodicInt2}),~(\ref{eq:ttprimLeqPi1et2}) and the previous display, we obtain 
\begin{align*}
\sup_{t\in\Rset}\sup_{t'\in\Rset}S(t,t',u)\leq&\sum_{p\in\Zset}(|2p\pi|+2\pi)^{-\delta}(|2p\pi-u|+2\pi)^{-\delta}\\
&+\sup_{t'\in\Rset}\sum_{p\in\Zset}
(|2t'+2p\pi| + 2\pi )^{-\delta}(|2t'+2p\pi+u| + 2\pi )^{-\delta}\\
&+\sup_{t'\in\Rset}\sum_{p\in\Zset}(|2p\pi|+2\pi)^{-\delta}(|2t'+2p\pi+u| + 2\pi )^{-\delta}\\
&+\sup_{t'\in\Rset}\sum_{p\in\Zset}(|2t'+2p\pi| + 2\pi )^{-\delta}(|2p\pi-u|+2\pi)^{-\delta}\;.
\end{align*}
Since 
the three functions in $t'$ appearing in the right-hand side of the last display 
are $\pi$-periodic the $\sup_{t'\in\Rset}$ can be replaced by $\sup_{|t'|\leq\pi/2}$. Since 
$|2t'+2p\pi|\geq |2p\pi|-\pi$ and $|2t'+2p\pi+u|\geq|2p\pi+u|-\pi$
for $|t'|\leq\pi/2$, we thus obtain
$$
\sup_{t\in\Rset}\sup_{t'\in\Rset}S(t,t',u)\leq 
4\sum_{p\in\Zset}(|2p\pi| + \pi )^{-\delta}(|2p\pi+u|+\pi )^{-\delta}\to 0 \quad\text{as $u\to\infty$}\;,
$$
which conclude the proof.
\end{proof}

The following lemma will be used in the proof of Lemma~\ref{lem:varlimitCLTlinear}.

\begin{lemma}\label{lem:empVarL2bound}
Let $p$ be a positive integer. For all $\Cset^p$-valued function $\bg\in L^2(-\pi,\pi)$ and $n\geq1$, define 
\begin{equation}\label{eq:Nn}
M_n(\bg)\eqdef \left\{\sum_{k\in\Zset} \left(1-\frac{|k|}{n}\right)_{\!\!\!+} |\bc_k|^2\right\}^{1/2} =
 \left\{\sum_{k=-n+1}^{n-1} \left(1-\frac{|k|}{n}\right) |\bc_k|^2\right\}^{1/2} 
\;,
\end{equation}
where $\bc_k=(2\pi)^{-1/2}\int_{-\pi}^\pi \bg(\lambda)\,\rme^{\rmi k \lambda}\,\rmd\lambda$ and  $|\cdot|$ denotes the Euclidean
norm in any dimension. Then, for all $\bg_1$ and $\bg_2$ in $L^2(-\pi,\pi)$,   
\begin{equation}\label{eq:NnLipschitz}
|M_n(\bg_1)-M_n(\bg_2)| \leq \left(\int_{-\pi}^\pi |\bg_1(\lambda)-\bg_2(\lambda)|^2\,\rmd\lambda\right)^{1/2} \; .
\end{equation}
Moreover, for all $\bg$ in $L^2(-\pi,\pi)$, as $n\to\infty$, 
\begin{equation}\label{eq:NnLim}
M_n(\bg) \to \left(\int_{-\pi}^\pi |\bg(\lambda)|^2\,\rmd\lambda\right)^{1/2} \; .
\end{equation}
\end{lemma}
\begin{proof}
See~\cite[Lemma~1 (Appendix~B)]{moulines-roueff-taqqu-2007a}.
\end{proof}
 
The following lemmas are used to compute the limiting covariances~(\ref{eq:GammaMDef}) and~(\ref{eq:GammaDef}).

\begin{lemma}\label{lem:DecompAjBj}
Let $\{Z_{i,j,k},\;i=1,2, j\geq0, k\in\Zset\}$ be an array of $2$--dimensional decimated linear processes as defined by~(\ref{eq:Zdef}).
Assume~\ref{it:A1}. Then for all $j\geq0$ and all $n\geq1$, one has
\begin{equation}
  \label{eq:AjBjDecomp}
  \frac1n\PCov\left(\sum_{k=0}^{n-1} Z_{1,j,k}^2, \sum_{k=0}^{n-1} Z_{2,j,k}^2\right)
  = 2 A_j(n) + \kappa_4 B_j(n) \; ,
\end{equation}
where
\begin{equation}
  \label{eq:AjDef}
A_j(n)=\sum_{\tau=-n+1}^{n-1} (1-|\tau|/n) \left(\sum_{u\in\Zset} \vseq_{1,j}(u)\vseq_{2,j}(\decim_j\tau+u)\right)^2 
\end{equation}
and
\begin{equation}
  \label{eq:BjDef}
B_j(n)=\sum_{u\in\Zset}\vseq_{1,j}^2(u)\, \sum_{\tau=-n+1}^{n-1} (1-|\tau|/n) \, \vseq_{2,j}^2(\decim_j\tau+u) \; .
\end{equation}
\end{lemma}
\begin{proof}
Using a standard formula for cumulants of products, we have
\begin{align*}
\PCov\left(\sum_{k=0}^{n-1} Z_{1,j,k}^2, \sum_{k=0}^{n-1} Z_{2,j,k}^2\right)
= &\sum_{k=0}^{n-1}  \sum_{k'=0}^{n-1} \PCov\left( Z_{1,j,k}^2,Z_{2,j,k'}^2\right)\\
= &2 \, \sum_{k=0}^{n-1}  \sum_{k'=0}^{n-1} \PCov^2 \left( Z_{1,j,k},Z_{2,j,k'}\right)\\
& + \sum_{k=0}^{n-1}  \sum_{k'=0}^{n-1} \mathrm{cum}\left(Z_{1,j,k},Z_{1,j,k},Z_{2,j,k'},Z_{2,j,k'}\right) \; .
\end{align*}
By definition of $\{Z_{i,j,k},\,i=1,2,k\in\Zset\}$ the covariance and the fourth-order cumulant in the previous
display  read respectively
$$
\PCov \left( Z_{1,j,k},Z_{2,j,k'}\right) =\sum_{t\in\Zset}\vseq_{1,j}(\decim_jk-t)\vseq_{2,j}(\decim_jk'-t)
$$
and
$$
\mathrm{cum}\left(Z_{1,j,k},Z_{1,j,k},Z_{2,j,k'},Z_{2,j,k'}\right)=
\kappa_4\sum_{t\in\Zset}\vseq_{1,j}^2(\decim_jk-t)\vseq_{2,j}^2(\decim_jk'-t) \; .
$$
The two last displays thus give~(\ref{eq:AjBjDecomp}).
\end{proof}

\begin{lemma}\label{lem:varlimitCLTlinear}
Let $A_j(n)$ and $B_j(n)$ be defined by~(\ref{eq:AjDef}) and~(\ref{eq:BjDef}), respectively, and $\vfonc_{i,j}$ by~(\ref{eq:vfoncdef}). 
Then the following inequalities hold for all $j\geq0$ and all $n\geq1$ :
\begin{align}
\label{eq:Ajboundvfonc}
A_j(n)&\leq 2\pi\int_{-\pi}^{\pi}\left|\decim_j^{-1}\sum_{p=0}^{\decim_j-1}
[\vfonc_{1,j}\overline{\vfonc_{2,j}}](\decim_j^{-1}(\lambda+2\pi p))\right|^2\,\rmd\lambda\\
\label{eq:Bjboundvfonc}
B_j(n) &\leq \int_{-\pi}^{\pi} \left|\vfonc_{1,j}(\lambda)\right|^2\,\rmd\lambda
\int_{-\pi}^{\pi} \left(\decim_j^{-1}\sum_{p=0}^{\decim_j-1} \left|\vfonc_{2,j}(\decim_j^{-1}(\lambda+2\pi p))\right|
\right)^2\,\rmd\lambda \; .
\end{align}
Suppose moreover that $\decim_j\to\infty$ as $j\to\infty$, $\decim_j$ is an even integer for $j$ large enough and that Condition~\ref{item:FourierConditions} holds for some 
$\lambda_{i,\infty}\in[0,\pi)$ and $\Rset\to\Zset$ functions $\vfonc_{i,\infty}$, $i=1,2$. 
Then, as $(n,j)\to(\infty,\infty)$,
\begin{align}
\label{eq:Ajlimit}
A_j(n) \to 2\pi \; \limcons_{1,2} \;\int_{-\pi}^{\pi} 
\left|\sum_{p\in\Zset} \vfoncsym_{1,2}(\lambda+2\pi p)\right|^2\,\rmd\lambda\;,
\end{align}
where $\limcons_{1,2}$ and $\vfoncsym_{1,2}$ are defined in~(\ref{eq:LimConsDef}) and~(\ref{eq:vfoncsymDef}) respectively.
Moreover,
\begin{align}
\label{eq:Bjlimit}
\lim_{j\to\infty} \sup_{n\geq1} \left|B_j(n)\right| = 0 \; .
\end{align}
\end{lemma}
\begin{proof}
\noindent\textbf{Step~1.} 
Using properties of the convolution of square summable sequences, we have, for all $t\in\Zset$,  
$$
\sum_{u\in\Zset} \vseq_{1,j}(u)\vseq_{2,j}(t+u)  = 
\int_{-\pi}^\pi \vfonc_{1,j}(\lambda) \overline{\vfonc_{2,j}(\lambda)}\,\rme^{\rmi\,t\lambda} \, \rmd\lambda \;.
$$
For any $\tau\in\Zset$, applying Lemma~\ref{lem:finiteFolding} to the $(2\pi)$-periodic function 
$\lambda\mapsto\vfonc_{1,j}(\lambda) \overline{\vfonc_{2,j}(\lambda)}\,\rme^{\rmi\,\decim_j\tau\lambda}$, one gets
$$
\sum_{u\in\Zset} \vseq_{1,j}(u)\vseq_{2,j}(\decim_j\tau+u) = (2\pi)^{-1/2} \int_{-\pi}^\pi 
\left((2\pi)^{1/2}\decim_j^{-1}\sum_{p=0}^{\decim_j-1} \vfonc_{1,j}(\xi_{j,p}(\lambda))\overline{\vfonc_{2,j}(\xi_{j,p}(\lambda))}\right)
\,\rme^{\rmi\,\tau\lambda} \, \rmd\lambda \;, 
$$
where 
\begin{equation}
  \label{eq:XiJPdef}
\xi_{j,p}(\lambda)\eqdef\decim_j^{-1}(\lambda+2\pi p) \;.
\end{equation}
Using the notation of Lemma~\ref{lem:empVarL2bound}, we can express $A_j(n)$ defined in~(\ref{eq:AjDef}), as
\begin{equation}
  \label{eq:AM}
 A_j(n)=\left(M_n(g_j)\right)^2
\end{equation}
where
\begin{equation}
  \label{eq:AjNewDef}
g_j(\lambda)=(2\pi)^{1/2}\decim_j^{-1}\sum_{p=0}^{\decim_j-1} \vfonc_{1,j}(\xi_{j,p}(\lambda))\overline{\vfonc_{2,j}(\xi_{j,p}(\lambda))}\;. 
\end{equation}
The bound~(\ref{eq:NnLipschitz}) in Lemma~\ref{lem:empVarL2bound} with $\bg_1=g_j$ and $\bg_2=0$ thus
gives~(\ref{eq:Ajboundvfonc}).

\noindent\textbf{Step~2.} 
Let us now show~(\ref{eq:Bjboundvfonc}). Since
$$
\vseq_{2,j}(\decim_j\tau+u) = (2\pi)^{-1/2}\;\int_{-\pi}^{\pi} \vfonc_{2,j}(\lambda)\;\rme^{\rmi\,\lambda(\decim_j\tau+u)}\rmd\lambda \; ,
$$
we can apply Lemma~\ref{lem:finiteFolding} to  the $(2\pi)$-periodic function 
$\lambda\mapsto\vfonc_{2,j}(\lambda)\;\rme^{\rmi\,\lambda(\decim_j\tau+u)}$ for all $u$ and $\tau$
in $\Zset$ and get
$$
\vseq_{2,j}(\decim_j\tau+u) = (2\pi)^{-1/2} \int_{-\pi}^{\pi} 
\left(\decim_j^{-1}\sum_{p=0}^{\decim_j-1} \vfonc_{2,j}(\xi_{j,p}(\lambda))\rme^{\rmi\,u\xi_{j,p}(\lambda)} \right)\,\rme^{\rmi\,\tau\lambda} \,
\rmd\lambda \; .
$$
Using the Parseval formula, we get, for all $u\in\Zset$,
\begin{align*}
\sum_{\tau=-n+1}^{n-1} (1-|\tau|/n) \, \vseq_{2,j}^2(\decim_j\tau+u)
& \leq \sum_{\tau\in\Zset}\vseq_{2,j}^2(\decim_j\tau+u) \\
& = \int_{-\pi}^{\pi} \left|\decim_j^{-1}\sum_{p=0}^{\decim_j-1} \vfonc_{2,j}(\xi_{j,p}(\lambda))\rme^{\rmi\,u\xi_{j,p}(\lambda)}
\right|^2\,\rmd\lambda \\
& \leq \int_{-\pi}^{\pi} \left(\decim_j^{-1}\sum_{p=0}^{\decim_j-1} \left|\vfonc_{2,j}(\xi_{j,p}(\lambda))\right|
\right)^2\,\rmd\lambda \; .
\end{align*}
Observing that the resulting bound is independent of $u\in\Zset$ and using the Parseval formula
$\sum_{u\in\Zset}\vseq_{1,j}^2(u) = \int_{-\pi}^{\pi} \left|\vfonc_{1,j}(\lambda)\right|^2\,\rmd\lambda$,
we obtain the bound~(\ref{eq:Bjboundvfonc}) for $B_j(n)$ defined in~(\ref{eq:BjDef}). 

\noindent\textbf{Step~3.} 
We now establish the limit~(\ref{eq:Ajlimit}) successively in the cases 
$\lambda_{1,\infty}\neq\lambda_{2,\infty}$ and $\lambda_{1,\infty}=\lambda_{2,\infty}$.
The $(2\pi)$-periodicity of $\vfonc_{1,j}\overline{\vfonc_{2,j}}$ and
Lemma~\ref{lem:finiteFolding} entail that $g_j$ (defined in~(\ref{eq:AjNewDef})) is $(2\pi)$-periodic.  By definition of
$M_n$ in 
Lemma~\ref{lem:empVarL2bound}, it follows that, for any $j\geq0$ and any $\tau\in\Rset$,
\begin{equation}
  \label{eq:AMbis}
 M_n(g_j)=M_n(g_j^{(\tau)})\quad \text{with}\quad
g_j^{(\tau)}(\lambda)=g_j(\lambda-\tau),\quad\lambda\in\Rset\;,
\end{equation}
since the modulus of the Fourier coefficients of $g_j$ and $g_j^{(\tau)}$ are equal.
In the following we will take $\tau=\pi\decim_j$. 
Observe that, for all $p\in\{0,\dots,\decim_j-1\}$,
$\lambda\in(0,2\pi)$ and $j\geq0$, 
\begin{equation}
  \label{eq:XijpIn-pipi}
  \xi_{j,p}(\lambda-\pi\decim_j)\in(-\pi,\pi) \; .
\end{equation}
Consider the case where
$\lambda_{1,\infty}\neq\lambda_{2,\infty}$, which, by~(\ref{eq:Limitfreq}), implies 
\begin{equation}
  \label{eq:limitfreqcasedifferent}
  \decim_j|\lambda_{1,j}-\lambda_{2,j}|\to\infty\quad\text{as $j\to\infty$}\;.
\end{equation}
Using~(\ref{eq:XijpIn-pipi}),~(\ref{eq:AjNewDef}),~(\ref{eq:AMbis}),~(\ref{eq:XijpIn-pipi}) 
and~(\ref{eq:unfiBoundvfoncSym}), we have, for some constant $C>0$, for all $j\geq0$,
\begin{align*}
\sup_{\lambda\in(0,2\pi)}\left|g_j^{(\pi\decim_j)}(\lambda)\right|&
=\sup_{\lambda\in(0,2\pi)}\left|g_j(\lambda-\pi\decim_j)\right|\\
&\leq C \sum_{p=0}^{\decim_j-1}
\prod_{i=0}^1(1+\decim_j\left||\xi_{j,p}(\lambda-\pi\decim_j)|-\lambda_{i,j}\right|)^{-\delta}\\
&\leq C\sup_{t,t'\in\Rset}\sum_{p\in\Zset}
(1+\left||t+2\pi p|-t'\right|)^{-\delta}\\
&\hspace{2cm}\times(1+\left||t+2\pi p|-t'-\decim_j|\lambda_{1,j}-\lambda_{2,j}|\right|)^{-\delta}\\
&\to 0 \quad\text{as $j\to\infty$},
\end{align*}
by~(\ref{eq:limitfreqcasedifferent}) and~(\ref{eq:supsupgjLim}) in
Lemma~\ref{lem:DomConvArgumenttt}. Applying~(\ref{eq:AM}),~(\ref{eq:AMbis}) and
the bound~(\ref{eq:NnLipschitz}) in Lemma~\ref{lem:empVarL2bound} with $\bg_1=g_j$ and $\bg_2=0$
yields
$$
A_j(n)=(M_n(g_j^{(\pi\decim_j)}))^2\leq
\int_{-\pi}^\pi|g_j^{(\pi\decim_j)}|^2\;\rmd\lambda\;.
$$
The last two displays and the $(2\pi)$-periodicity of $g_j$ imply
$A_j(n)\to0$ as $j\to\infty$. This proves~(\ref{eq:Ajlimit}) since by~(\ref{eq:LimConsDef}), $C_{1,2}=0$ when $\lambda_{1,\infty}\neq\lambda_{2,\infty}$.

We now consider the case $\lambda_{1,\infty}=\lambda_{2,\infty}$. By Condition~(\ref{eq:CondFreqEqualOrDifferent}), we have
$\lambda_{1,j}=\lambda_{2,j}$ for all $j$ large enough. 
Let $p_j= \decim_j\lambda_{1,j}/(2\pi)+\decim_j/2$ so that 
\begin{equation}
  \label{eq:lambda01pj}
\lambda_{1,j}=\lambda_{2,j}=2\pi\decim_j^{-1} p_j-\pi\;. 
\end{equation}
By Condition~(\ref{eq:integerCondition}) and since $\decim_j$ is even for $j$ large enough by assumption, we get that $p_j$ is an integer for $j$ large enough. 
Writing
$$
\sum_{p=0}^{\decim_j-1}=
\sum_{p=0}^{[\decim_j/2]-1}+\sum_{p=[\decim_j/2]}^{\decim_j-1}
=\sum_{q=-(\decim_j-p_j)}^{[\decim_j/2]-(\decim_j-p_j)-1}+\sum_{r=[\decim_j/2]-p_j}^{\decim_j-p_j-1}\;,
$$
where $q=p-(\decim_j-p_j)$ and $r=p-p_j$ and observe that, with these definitions,~(\ref{eq:XiJPdef})
and~(\ref{eq:lambda01pj}), we have $\xi_{j,q}(\lambda)=\decim_j^{-1}(\lambda+2\pi q)=
\decim_j^{-1}[(\lambda-\pi\decim_j)+2\pi p]+2\pi\decim_j^{-1} p_j-\pi
=\xi_{j,p}(\lambda-\pi\decim_j)+\lambda_{1,j}$, and, similarly, 
$\xi_{j,q}(\lambda)=\xi_{j,p}(\lambda-\pi\decim_j)-\lambda_{1,j}$, so that 
$g_j^{(\pi\decim_j)}(\lambda)=g_j(\lambda-\pi\decim_j)$ defined in~(\ref{eq:AMbis}) and~(\ref{eq:AjNewDef}), can be expressed as
\begin{multline}\label{eq:gjinTwoParts}
g_j^{(\pi\decim_j)}(\lambda)=(2\pi)^{1/2}\left[
\sum_{q=-(\decim_j-p_j)}^{[\decim_j/2]-(\decim_j-p_j)-1}
\frac{\vfonc_{1,j}\overline{\vfonc_{2,j}}}{\decim_j}(\xi_{j,q}(\lambda)-\lambda_{1,j})\right.\\
\left.+\sum_{r=[\decim_j/2]-p_j}^{\decim_j-p_j-1}
\frac{\vfonc_{1,j}\overline{\vfonc_{2,j}}}{\decim_j}(\xi_{j,r}(\lambda)+\lambda_{1,j})
\right]\;. 
\end{multline}
Since $\lim_{j\to\infty}\gamma_j=\infty$ and, by Condition~(\ref{eq:Limitfreq}),
$\lim_{j\to\infty}\lambda_{i,j}=\lambda_{i,\infty}\in[0,\pi)$, we have $\lim_{j\to\infty} \decim_j^{-1} p_j \in[1/2,1)$ and thus
\begin{equation}
  \label{eq:limitpjSurDecimjSansCond}
  -(\decim_j-p_j)\to-\infty\quad\text{and}\quad\decim_j-p_j-1\to\infty\;,
\end{equation}
namely, that in~(\ref{eq:gjinTwoParts}), the upper limit of the first sum tends to $\infty$ and the bottom limit of the second
sum  tends to $-\infty$. We now consider the remaining limits.
If $\lambda_{1,\infty}=\lambda_{2,\infty}>0$, then
$\lim_{j\to\infty} \decim_j^{-1} p_j$ falls in the open interval $(1/2,1)$ and thus 
\begin{equation}
  \label{eq:limitpjSurDecimj}
[\decim_j/2]-(\decim_j-p_j)-1\to\infty\quad\text{and}\quad[\decim_j/2]-p_j\to-\infty,\;.
\end{equation}
If $\lambda_{1,\infty}=\lambda_{2,\infty}=0$, using~(\ref{eq:CondFreqEqual0orDifferent}),~(\ref{eq:lambda01pj}) implies
$p_j=\decim_j/2$ and thus, for $j$ large enough so that $p_j$ is integer--valued and $\decim_j$ even,
\begin{equation}
  \label{eq:limitpjSurDecimjFreqZero}
[\decim_j/2]-(\decim_j-p_j)-1=-1\quad\text{and}\quad[\decim_j/2]-p_j=0\;.
\end{equation}
In view of~(\ref{eq:XiJPdef}), Conditions~(\ref{eq:CondFreqEqualOrDifferent}),~(\ref{eq:unfiBoundvfonc}) and~(\ref{eq:Limitvfonc}) 
(which imply~(\ref{eq:LimitvfoncAltern}),~(\ref{eq:unfiBoundvfoncSym})
and~(\ref{eq:LimitvfoncSym})),~(\ref{eq:limitpjSurDecimjSansCond}),~(\ref{eq:limitpjSurDecimj}),~(\ref{eq:limitpjSurDecimjFreqZero})
and dominated convergence yield, for all $\lambda\in(0,2\pi)$,    
\begin{equation}\label{eq:limiteCov2simple}
g_j^{(\pi\decim_j)}(\lambda)\to g_\infty(\lambda)\quad\text{as $j\to\infty$}\;,
\end{equation}
where
$$
g_\infty(\lambda)=(2\pi)^{1/2}
\left[\sum_{q\in\Zset}[\overline{\vfonc_{1,\infty}}\vfonc_{2,\infty}](-\lambda-2\pi q)
+\sum_{r\in\Zset}
[\vfonc_{1,\infty}\overline{\vfonc_{2,\infty}}](\lambda+2\pi r)\right]
$$
if $\lambda_{1,\infty}=\lambda_{2,\infty}>0$, and
$$
g_\infty(\lambda)=(2\pi)^{1/2}
\left[\sum_{q=-\infty}^{-1}[\overline{\vfonc_{1,\infty}}\vfonc_{2,\infty}](-\lambda-2\pi q)
+\sum_{r=0}^\infty
[\vfonc_{1,\infty}\overline{\vfonc_{2,\infty}}](\lambda+2\pi r)\right]
$$
if $\lambda_{1,\infty}=\lambda_{2,\infty}=0$.
By definition of $\vfoncsym$ in~(\ref{eq:vfoncsymDef}) and using~(\ref{eq:symetryFreqZero}),
 one has
$[\overline{\vfonc_{1,\infty}}\vfonc_{2,\infty}](-\lambda)=
[\vfonc_{1,\infty}\overline{\vfonc_{2,\infty}}](\lambda)=\vfoncsym_{1,2}(\lambda)$,  
the two previous displays read 
\begin{equation}
\label{eq:gByvfonc}
g_\infty(\lambda)= (2\pi)^{1/2}\; \limcons_{1,2}\;
\sum_{p\in\Zset} \vfoncsym_{1,2}(\lambda+2\pi p) \;.
\end{equation}
Conditions~(\ref{eq:unfiBoundvfonc}) and~(\ref{eq:Limitvfonc}) imply~(\ref{eq:LimitvfoncBound}), thus that $g_\infty(\lambda)$ is bounded, hence square integrable on $\lambda\in(-\pi,\pi)$. 
Moreover, applying the same dominated argument as above, one has
\begin{align}\label{eq:limiteCov2}
\lim_{j\to0}
\int_{-\pi}^{\pi}
\left|g_j^{(\pi\decim_j)}(\lambda)-g_\infty(\lambda)\right|^2
\,\rmd\lambda = 0 \; .
\end{align}
One gets by~(\ref{eq:AMbis}),~(\ref{eq:NnLipschitz}) and~(\ref{eq:limiteCov2})
\begin{equation}
  \label{eq:limMinJ}
\left|M_n(g_j)-M_n(g_\infty)\right|^2
\leq \int_{-\pi}^{\pi}
\left|g_j(\lambda)-g_\infty(\lambda)\right|^2
\,\rmd\lambda \to0\quad\text{as $j\to\infty$.}
\end{equation}
By applying the limit~(\ref{eq:NnLim}) with $\bg=g_\infty$, one gets 
\begin{equation}
  \label{eq:limMinN}
  M_n(g_\infty)^2\to\int_{-\pi}^\pi|g_\infty(\lambda)|^2\;\rmd\lambda      
\end{equation}
as $n\to\infty$. 
Hence, setting $M_n(g_j)= \left(M_n(g_j)-M_n(g_\infty)\right)+M_n(g_\infty)$,
the limit~(\ref{eq:Ajlimit}) follows from~(\ref{eq:AM}),~(\ref{eq:gByvfonc}),~(\ref{eq:limMinJ}) and~(\ref{eq:limMinN}).     

\noindent\textbf{Step~4.} 
We now establish the limit~(\ref{eq:Bjlimit}). 
By~Condition~(\ref{eq:unfiBoundvfonc}), we have
\begin{equation}
  \label{eq:varianceZBound}
\sup_{j\geq0}\int_{-\pi}^{\pi} \left|\vfonc_{1,j}(\lambda)\right|^2\,\rmd\lambda < \infty \; .  
\end{equation}
Using similar arguments as above and Condition~(\ref{eq:unfiBoundvfonc}), we have
\begin{multline}\label{eq:OtherBjPartBound} 
\int_{-\pi}^{\pi} \left(\decim_j^{-1}\sum_{p=0}^{\decim_j-1} 
\left|\vfonc_{2,j}(\xi_{j,p}(\lambda))\right|\right)^2\,\rmd\lambda 
= \decim_j^{-1}\int_{0}^{2\pi} 
\left(\sum_{p=0}^{\decim_j-1}\decim_j^{-1/2}\left|\vfonc_{2,j}(\xi_{j,p}(\lambda-\pi\decim_j))\right|
\right)^2\,\rmd\lambda \\
\leq C\; \decim_j^{-1}
\int_{0}^{2\pi} \left(\sum_{p=0}^{\decim_j-1}
(1+\left||\lambda+2\pi(p-\decim_j/2)|-\decim_j\lambda_{2,j}\right|)^{-\delta}\right)^2\,\rmd\lambda
\end{multline}
Using that $||a+b|-c|\geq||b|-c|-|a|$ and  $\lambda\in[0,2\pi]$, we have
$$
\left||\lambda+2\pi(p-\decim_j/2)|-\decim_j\lambda_{2,j}\right|
\geq \left|2\pi|p-\decim_j/2|-\decim_j\lambda_{2,j}\right|- 2\pi
$$
Take $j$ large enough so that $\decim_j$ is even.
Since $\lambda_{2,j}\in[0,\pi)$, as $p\in\{0,\dots,\decim_j-1\}$, 
$2\pi|p-\decim_j/2|-\decim_j\lambda_{2,j}$ is a sequence of numbers with lag $2\pi$ and
belonging to $[-\decim_j\pi,\decim_j\pi]$
and can thus be written as a sequence $2\pi q+ c$, where $q$ belongs to $\{-\decim_j/2,\dots,\decim_j/2\}$
and $c$ to $[-\pi,\pi]$ so that
$$
\left|2\pi|p-\decim_j/2|-\decim_j\lambda_{2,j}\right|- 2\pi \geq 2\pi |q| - 3\pi \; .
$$
From the last two displays, we have
$$
\left(5\pi+\left||\lambda+2\pi(p-\decim_j/2)|-\decim_j\lambda_{2,j}\right|\right)^{-\delta}
\leq (2\pi)^{-\delta}(1+|q|)^{-\delta}\;,
$$
with $q$ describing $\{-\decim_j/2,\dots,\decim_j/2\}$ as $p$ describes $\{0,\dots,\decim_j-1\}$. 
Inserting this bound in~(\ref{eq:OtherBjPartBound}), we get
$$
\int_{-\pi}^{\pi} \left(\decim_j^{-1}\sum_{p=0}^{\decim_j-1} 
\left|\vfonc_{2,j}(\xi_{j,p}(\lambda))\right|\right)^2\,\rmd\lambda 
 \leq C\; \decim_j^{-1} \left(\sum_{q=0}^{\decim_j/2}(1+q)^{-\delta}\right)^2
$$
for some constant $C$ not depending on $j\geq0$. Since the last right-hand side of the previous display tends to 0 as
$j\to\infty$ for any $\delta>1/2$, with~(\ref{eq:varianceZBound}) and~(\ref{eq:Bjboundvfonc}), we obtain~(\ref{eq:Bjlimit}).
\end{proof}

\begin{remark}
  The factor $4\pi=2\times2\pi$ in~(\ref{eq:GammaDef}) is due to the factor 2 in the right-hand side
  of~(\ref{eq:AjBjDecomp}) and the presence of $2\pi$ in the  right-hand side
  of~(\ref{eq:Ajlimit}). 
\end{remark}

\begin{corollary}\label{cor:Limcov}
Let $\{Z_{i,j,k},\;i=1,2, j\geq0, k\in\Zset\}$ be an array of $2$--dimensional decimated linear processes as defined by~(\ref{eq:Zdef}).
Assume~\ref{it:A1}, that $\decim_j$ is even for $j$ large enough and that
Condition~\ref{item:FourierConditions} holds for some $\lambda_{i,\infty}\in[0,\pi)$ and $\Rset\to\Zset$ functions $\vfonc_{i,\infty}$,
$i=1,2$. Then, for all $k,k'\in\Zset$, as $j\to\infty$,
\begin{equation}
  \label{eq:LimcovSquare-neq1}
\PCov\left(Z_{1,j,k}^2, Z_{2,j,k'}^2\right)\to
2\,\limcons_{1,2}^2\,\left(\int_{-\infty}^{\infty}\vfoncsym_{1,2}(\lambda)\rme^{\rmi\lambda(k'-k)}\,\rmd\lambda\right)^2\;,
\end{equation}
where $\limcons_{1,2}$ and $\vfoncsym_{1,2}$ are defined in~(\ref{eq:LimConsDef}) and~(\ref{eq:vfoncsymDef}), 
and, as $(n,j)\to(\infty,\infty)$, 
\begin{equation}
  \label{eq:Limcov}
  \frac1n\PCov\left(\sum_{k=0}^{n-1} Z_{1,j,k}^2, \sum_{k=0}^{n-1} Z_{2,j,k}^2\right)\to
4\pi\,\limcons_{1,2}\,\int_{-\pi}^{\pi} 
\left|\sum_{p\in\Zset} \vfoncsym_{1,2}(\lambda+2\pi p)\right|^2\,\rmd\lambda\;.
\end{equation}
\end{corollary}
\begin{proof}
Setting $n=1$ in~(\ref{eq:AjBjDecomp}) and~(\ref{eq:AjDef}) and replacing $\vseq_{1,j}(-t)$ by $\vseq_{1,j}(\decim_jk-t)$ and $\vseq_{2,j}(-t)$ by $\vseq_{2,j}(\decim_jk'-t)$
so that $Z_{1,j,0}$ is replaced by $Z_{1,j,k}$ and $Z_{2,j,0}$  by $Z_{2,j,k'}$, we get 
\begin{align*}
\PCov\left(Z_{1,j,k}^2,Z_{2,j,k'}^2\right)
&=2A_j(1)+\kappa_4B_j(1)\\
&=2\left(\sum_{t\in\Zset} \vseq_{1,j}(\decim_jk-t)\vseq_{2,j}(\decim_jk'-t)\right)^2 +\kappa_4B_j(1)\;,
\end{align*}
and thus~(\ref{eq:LimcovSquare-neq1}) follows from~(\ref{eq:Bjlimit}),~(\ref{eq:covZ0Z1sum}) and~(\ref{eq:Limcov-neq1}).

Relation~(\ref{eq:Limcov}) is obtained by applying Lemmas~\ref{lem:DecompAjBj} and~\ref{lem:varlimitCLTlinear}. 
\end{proof}
\section{Proof of Theorem~\ref{thm:CLTlinear}}
\label{sec:proof-theor-clt}



We first establish two Central Limit Theorems which will be used in the proof of Theorem~\ref{thm:CLTlinear}.
The first involves a sequence of linear filters of the sequence $\{\xi_t,\,t\in\Zset\}$. 

\begin{proposition}\label{prop:CLTlinearForms} 
Define, for all $i=1,\dots,\dimV$, $j\geq0$,  $k\in\Zset$,
\begin{equation} 
\label{eq:ZdefNotime}
Z_{i,j} = \sum_{t\in\Zset} \vseq_{i,j}(t) \, \xi_t  \;,
\end{equation}
where for all $i=1,\dots,\dimV$ and $j\geq0$,
$\{\vseq_{i,j}(t),\;t\in\Zset\}$ is  real-valued and satisfies $\sum_{t\in\Zset} \vseq_{i,j}^2(t)<\infty$ and
$\{\xi_t,\;t\in\Zset\}$ satisfies~\ref{it:A1-2moments}. Assume that
\begin{align}
\label{eq:LimitvSecSup}
&\lim_{j\to\infty}\sup_{t\in\Zset} |\vseq_{i,j}(t)| = 0 \quad\text{for all}\quad i=1,\dots,\dimV\;, \\
\label{eq:LimitvSecCov}
&\lim_{j\to\infty}\sum_{t\in\Zset} \left(\vseq_{i,j}(t)\vseq_{i',j}(t)\right) = \Sigma_{i,i'} \quad\text{for all}\quad
i,i'=1,\dots,\dimV\;,
\end{align}
where $\Sigma$ is a $\dimV\times\dimV$ given matrix. Then, as $j\to\infty$,
\begin{equation}
\label{eq:CLT_Z_alone}
\left[
\begin{array}{c}
Z_{1,j} \\
\vdots\\
Z_{\dimV,j}
\end{array}
\right]\cl\calN(0,\Sigma) \; .
\end{equation}
\end{proposition}
\begin{remark}
A study on the weak convergence of such sequence without assuming Assumption~\ref{it:A1-2moments} can be found
in~\cite{lang:soulier:1998}.
\end{remark}

\begin{proof}
This is a standard application of the Lindeberg-Feller Central Limit Theorem. 
Using the Cramér-Wold device for vectorial Central Limit Theorem and since for any column vector $\zeta \in\Rset^{\dimV}$, the linear
combination $\sum_{i=1}^\dimV\zeta_iZ_{i,j}$ can 
be written as in~(\ref{eq:ZdefNotime}) with $\vseq_{i,j}(t)$ replaced by
$\vseq_{j}(t;\zeta)=\sum_{i=1}^\dimV\zeta_i\vseq_{i,j}(t)$, which satisfies
\begin{align*}
&\lim_{j\to\infty}\sup_{t\in\Zset} |\vseq_{j}(t;\zeta)| = 0 \quad\text{for all}\quad i=1,\dots,\dimV\;, \\
&\lim_{j\to\infty}\sum_{t\in\Zset} \left(\vseq_{j}(t;\zeta)\vseq_{j}(t;\zeta)\right) = \zeta^T\Sigma\zeta\;,
\end{align*}
it is sufficient by~(\ref{eq:LimitvSecSup}) and~(\ref{eq:LimitvSecCov}) to prove the result for $\dimV=1$, in which case we simply denote $\vseq_{1,j}(t)$ by  $\vseq_{j}(t)$ and
$\Sigma$ by $\sigma^2$. Let
$(m_j)$ be a sequence of integers tending to infinity with $j$, such that 
$$
\sum_{|t|\geq m_j}\vseq_{j}^2(t)\leq 2^{-j} \;.
$$
We now show that the Lindeberg conditions hold for the sequence $\sum_{|t|\leq m_j} \vseq_{j}(t) \xi_t$. The first holds because,
by~(\ref{eq:LimitvSecCov}), 
$$
\lim_{j\to\infty}\sum_{|t|\leq m_j}\PVar(\vseq_{j}(t) \xi_t)=\lim_{j\to\infty}\sum_{|t|\leq m_j}\vseq_{j}^2(t)\eqdef\sigma^2\;. 
$$
The second holds because, for all $\epsilon>0$,
$$
\sum_{|t|\leq m_j}\PE\left[\vseq_{j}^2(t) \xi^2_t\;\1(|\vseq_{j}(t) \xi_t|\geq\epsilon)\right]
\leq\left(\sum_{t\in\Zset}\vseq_{j}^2(t)\right)\PE\left[\xi^2_0\;\1(|\xi_0|\geq\epsilon/S_j)\right]\;,
$$
where $S_j=\sup_{t\in\Zset} |\vseq_{j}(t)|$, and, by~(\ref{eq:LimitvSecSup}), the right-hand side of the last display
tends to 0 as $j\to\infty$. 
This concludes the proof. 
\end{proof}

The second Central Limit Theorem deals with  $m$-dependent arrays.
Recall that $\{Y_k\}$ is said to be $m$-dependent if, for all $p\geq1$ and all  
$k_1,\dots,k_p$ such that  $k_1+m\leq k_2,\dots,k_{p-1}+m\leq k_p$, $Y_{k_1},\dots,Y_{k_p}$ are independent. 

\begin{proposition}\label{prop:cltMdep}
Let $m$ be a fixed integer and $(n_j)$ a sequence of integers such that $n_j\to\infty$ as $j\to\infty$. Let $\{ Y_{j,k},\;
k=0,\dots,n_j-1,\, j\geq1\}$ be an array of $\Rset^d$-valued random vectors, such that, for each $j\geq0$, $\{ Y_{j,k},\;
k=0,\dots,n_j-1\}$ has zero mean, is strictly stationary and $m$-dependent. Assume that there exists a $d\times d$ matrix 
$\Gamma$ and a centered stationary $\Rset^d$-valued process $\{\bY_k,\; k\geq0\}$ with finite variance such that
\begin{align}
\label{eq:CLTMdep-LimDist}
&Y_{j,\centerdot} \cl \bY_\centerdot\quad\text{as}\quad j\to\infty\;, \\
\label{eq:CLTMdep-LimVar}
&\lim_{j\to\infty}  \PCov\left(Y_{j,k},Y_{j,k'}\right) = \PCov\left(\bY_{k},\bY_{k'}\right) \quad\text{for all}\quad k,k'\geq0 \\
  \label{eq:CLTMdep-LimVarLJ}
&\lim_{l\to\infty}\lim_{j\to\infty}\PCov\left(l^{-1/2}\sum_{k=0}^{l-1}Y_{j,k}\right) = \Gamma \;.
\end{align}
Then we have, as $j\to\infty$,
\begin{equation}
  \label{eq:CLTMDEP}
n_j^{-1/2}\sum_{k=0}^{n_j-1}Y_{j,k}\cl\calN(0,\Gamma) \; .
\end{equation}
\end{proposition}
\begin{proof}
We may suppose that $d=1$ since the vector case follows by the Cramér-Wold device. For convenience, we set $\Gamma=\sigma^2$.    
Let $s$ be a positive integer larger than $m$.
We decompose of $\sum_{k=0}^{n_j-1}Y_{j,k}$ in sums of random variables spaced by $m$, as follows:
\begin{equation}
  \label{eq:MdepDecomp}
\sum_{k=0}^{n_j-1}Y_{j,k}=\sum_{k=0}^{p_j}S_{j,k}^{(s)}  + \sum_{k=0}^{p_j}T_{j,k}^{(s)}+ R_j^{(s)}\;,  
\end{equation}
where 
$$
S_{j,k}^{(s)}= \sum_{i=0}^{s-1}Y_{j,k(m+s)+i},\quad
T_{j,k}^{(s)}= \sum_{i=0}^{m-1}Y_{j,k(m+s)+s+i}
\quad\text{and}\quad R_j^{(s)}=\sum_{i=0}^{q_j-1}Y_{j,p_j(m+s)+i} \;,
$$
and where 
$p_j$ and $q_j$ are the non-negative integers defined by the Euclidean division 
\begin{equation}
  \label{eq:euclide}
  n_j=p_j(m+s)+q_j\quad\text{with}\quad q_j\in\{0,\dots,m+s-1\}\;.
\end{equation} 

The $m$-dependence and the strict stationarity of the sequences $Y_{j,\centerdot}$ ensure
that for all $j\geq0$ and all $s\geq m$, $S_{j,\centerdot}^{(s)}$ and $T_{j,\centerdot}^{(s)}$ are sequences of centered independent and
identically distributed random variables. Hence, by~(\ref{eq:CLTMdep-LimVar}), we have
\begin{equation}
  \label{eq:Lindeberg-mDep-Cond1}
\lim_{j\to\infty}\sum_{k=0}^{p_j}\PVar\left(p_j^{-1/2}S_{j,k}^{(s)} \right)=
\lim_{j\to\infty}\PVar\left(S_{j,0}^{(s)} \right)=\PVar\left(\bS^{(s)}\right),
\end{equation}
where
\begin{equation*}
\bS^{(s)}\eqdef\sum_{i=0}^{s-1}\bY_{i} \;,
\end{equation*}
\begin{equation}
  \label{eq:Lindeberg-mDep-TjVar}
\lim_{j\to\infty}\PVar\left(p_j^{-1/2}\sum_{k=0}^{p_j}T_{j,k}^{(s)} \right) =
\lim_{j\to\infty}\PVar\left(T_{j,0}^{(s)} \right) = 
\PVar\left(\bS^{(m)}\right)
\end{equation}
and
\begin{equation}
  \label{eq:Lindeberg-mDep-RjVar}
\limsup_{j\to\infty}\PVar\left(R_j^{(s)}\right) \leq \max\left\{\PVar\left(\bS^{(t)}\right)\;:\;t=0,1,\dots,m+s-1\right\} <\infty \; .
\end{equation}
In addition, for any $\epsilon>0$,
$$
\sum_{k=0}^{p_j}\PE\left[(p_j^{-1/2}S_{j,k}^{(s)})^2\1(p_j^{-1/2}|S_{j,k}^{(s)}|>\epsilon) \right]
= \PE\left[(S_{j,0}^{(s)})^2\1(p_j^{-1/2}|S_{j,k}^{(s)}|>\epsilon) \right]
$$
and hence, since $p_j\to\infty$,
\begin{multline}\label{eq:Lindeberg-mDep-Cond2Bound}
\limsup_{j\to\infty}\sum_{k=0}^{p_j}\PE\left[(p_j^{-1/2}S_{j,k}^{(s)})^2\1(p_j^{-1/2}|S_{j,k}^{(s)}|>\epsilon) \right]\\
\leq \inf_{M>0} \limsup_{j\to\infty}\PE\left[(S_{j,0}^{(s)})^2\1(|S_{j,k}^{(s)}|>M) \right] \; .
\end{multline}
But using~(\ref{eq:CLTMdep-LimDist}), we have 
$$
S_{j,0}^{(s)}\cl \bS^{(s)}\quad\text{as}\quad j\to\infty\;.
$$
Hence, denoting by $\phi_M$ some continuous $\Rset_+\to[0,1]$ function satisfying $\1(x\leq M/2)\leq \phi_M(x) \leq\1(x\leq
M)$, so that $x^2\phi_M(x)$ is continuous and bounded, we have 
\begin{align*}
  \PE\left[(S_{j,0}^{(s)})^2\1(|S_{j,k}^{(s)}|\leq M) \right]
&\geq\PE\left[(S_{j,0}^{(s)})^2\phi_M(S_{j,0}^{(s)})\right]\\
&\to \PE\left[(\bS^{(s)})^2\phi_M(\bS^{(s)})\right]\quad\text{as}\quad j\to\infty\\
&\to \PE\left[(\bS^{(s)})^2\right]
\quad\text{as}\quad M\to\infty\; .
\end{align*}
Using~(\ref{eq:CLTMdep-LimVar}), we have
$\PE\left[(S_{j,0}^{(s)})^2\right]\to\PE\left[(\bS^{(s)})^2\right]$ as $j\to\infty$ and hence, for any $M>0$,
$$
 \limsup_{j\to\infty}\PE\left[(S_{j,0}^{(s)})^2\1(|S_{j,k}^{(s)}|>M) \right] 
= \PE\left[(\bS^{(s)})^2\right] - \liminf_{j\to\infty}\PE\left[(S_{j,0}^{(s)})^2\1(|S_{j,k}^{(s)}|\leq M) \right]
$$
The two last displays and~(\ref{eq:Lindeberg-mDep-Cond2Bound}) imply the second Lindeberg Condition, namely,
$$
\limsup_{j\to\infty}\sum_{k=0}^{p_j}\PE\left[(p_j^{-1/2}S_{j,k}^{(s)})^2\1(p_j^{-1/2}|S_{j,k}^{(s)}|>\epsilon) \right] =0 
\quad\text{for any $\epsilon>0$}\;.
$$
Using this and~(\ref{eq:Lindeberg-mDep-Cond1}), we may apply
the Lindeberg-Feller CLT for arrays of independent r.v.'s and we obtain
\begin{equation}
  \label{eq:limbS1}
  p_j^{-1/2}\sum_{k=0}^{p_j}S_{j,k}^{(s)} \cl \calN\left(0,\PVar\left(\bS^{(s)}\right)\right) \; .  
\end{equation}
By~(\ref{eq:CLTMdep-LimVarLJ}), we have
\begin{equation}
  \label{eq:limbS2}
\lim_{s\to\infty}s^{-1} \PVar\left(\bS^{(s)}\right) = \sigma^2 \; .
\end{equation}
Using~(\ref{eq:limbS1}) and~(\ref{eq:euclide}) with $j\to\infty$ and then~(\ref{eq:limbS2}) with $s\to\infty$ yields  
$$
n_j^{-1/2}\sum_{k=0}^{p_j}S_{j,k}^{(s)} \cl_{j\to\infty} \calN\left(0,(m+s)^{-1}\PVar\left(\bS^{(s)}\right)\right)
\cl_{s\to\infty}\calN\left(0,\sigma^2\right)\;.
$$
On the other hand, by~(\ref{eq:MdepDecomp}),~(\ref{eq:Lindeberg-mDep-TjVar}) and~(\ref{eq:Lindeberg-mDep-RjVar}), we have 
$$
\limsup_{j\to\infty}\PE\left[\left(n_j^{-1/2}\sum_{k=0}^{n_j-1}Y_{j,k}-n_j^{-1/2}\sum_{k=0}^{p_j}S_{j,k}^{(s)}\right)^2 \right]
\leq (m+s)^{-1}\PVar\left(\bS^{(m)}\right)\to 0
$$
as $s\to\infty$. Using the last two displays and~\cite[Theorem~3.2]{billingsley:1999}, we obtain~(\ref{eq:CLTMDEP}), which
concludes the proof.
\end{proof}

\begin{proof}[Proof of Theorem~\ref{thm:CLTlinear}]
The proof is in three steps. We show in a first step the convergence of the process 
$[Z_{i,j,k}, i=1,\dots,\dimV,\;k=0,\dots,n_j]$ as $j\to\infty$ towards a Gaussian limit.
In the second step we prove Theorem~\ref{thm:CLTlinear} under the additional assumption that the sequence
$$
Y_{j,k}\eqdef\left[
\begin{array}{c}
\{Z_{1,j,k}^2 -\PE[Z_{1,j,k}^2]\} \\
\vdots\\
\{Z_{\dimV,j,k}^2 -\PE[Z_{\dimV,j,k}^2]\}
\end{array}
\right],\quad k=0,\dots,n_j-1
$$
is $m$-dependent. The third step exhibits an $m$-dependent approximation and extends the $m$-dependent case to the general
case. The proof uses a number of auxiliary results proved in Section~\ref{sec:technical-lemmas}.

\noindent\textbf{Step~1.} We shall apply Proposition~\ref{prop:CLTlinearForms}. By
Relation~(\ref{eq:Limcov-neq1}) in Proposition~\ref{prop:Limcov}, we get, for
all $i,i'=1,\dots,\dimV$ and all $k,k'\in\Zset$, as $j\to\infty$,
\begin{equation}
  \label{eq:CovarianceZ}
  \PCov\left(Z_{i,j,k},Z_{i',j,k'}\right)
\to \limcons_{i,i'}\;
\int_{-\infty}^{\infty}\vfoncsym_{i,i'}(\lambda)\;\rme^{\rmi \lambda(k'-k)}\,\rmd\lambda \;.
\end{equation}
Moreover, by~(\ref{eq:unfiBoundvfonc}), one has, for all $i=1,\dots,\dimV$,
\begin{align*}
\sup_{t\in\Zset}\left|\vseq_{i,j}(t)\right|
&\leq 2\;(2\pi)^{-1/2}\int_{0}^{\pi} \left|\vfonc_{i,j}(\lambda)\right|\,\rmd\lambda \\
&\leq 2\;(2\pi)^{-1/2}\decim_j^{-1/2}\;\int_{-\decim_j\lambda_{i,j}}^{\decim_j(\pi-\lambda_{i,j})}(1+|\lambda|)^{-\delta}\,\rmd\lambda \;,
\end{align*}
which, by~(\ref{eq:Limitfreq}), tends to 0 as $j\to\infty$ for any $\delta>1/2$. Hence, by Proposition~\ref{prop:CLTlinearForms}, for any $p\geq1$, any
$i_1,\dots, i_p\in\{1,\dots,\dimV\}$ and any $k_1,\dots,k_p\in\Zset$, we have, as $j\to\infty$,
\begin{equation}
  \label{eq:CLTZaloneVect}
\left[
\begin{array}{c}
Z_{i_1,j,k_1} \\
\vdots\\
Z_{i_p,j,k_p}
\end{array}
\right]\cl\calN(0,\Sigma) 
\end{equation}
where $\Sigma$ is the covariance matrix with entries  given for all $1\leq n, n'\leq p$ by
\begin{align*}
\Sigma_{n,n'}&=
 \limcons_{i_n,i_{n'}}\;
\int_{-\infty}^{\infty}\vfoncsym_{i_n,i_{n'}}(\lambda)\;\rme^{\rmi \lambda(k_{n'}-k_n)}\,\rmd\lambda \;.
\end{align*}
Expressing this integral as $\sum_p\int_{-\pi+2p\pi}^{\pi+2p\pi}$, we get
\begin{align*}
\Sigma_{n,n'} = \limcons_{i_n,i_{n'}}\;
\int_{-\pi}^{\pi}\left(\sum_{p\in\Zset}\vfoncsym_{i_n,i_{n'}}(\lambda+2p\pi)\right)
\;\rme^{\rmi \lambda(k_{n'}-k_n)}\,\rmd\lambda \; .
\end{align*}
The convergence~(\ref{eq:CLTZaloneVect}) can be written equivalently as
\begin{equation}
  \label{eq:CLTZaloneProc}
\left[
\begin{array}{c}
Z_{1,j,\centerdot} \\
\vdots\\
Z_{\dimV,j,\centerdot}
\end{array}
\right]\cl\bZ_\centerdot \;,
\end{equation}
as $j\to\infty$, where $\{\bZ_k=[Z_{1,\infty,k}\,\,\dots\,\,Z_{\dimV,\infty,k}]^T,\;k\geq0\}$ is a stationary Gaussian $\Rset^{\dimV}$-valued process with spectral density matrix function $D$
with entries
$$
D_{i,i'}(\lambda)= \limcons_{i,i'}\;
\sum_{p\in\Zset}\vfoncsym_{i_n,i_{n'}}(\lambda+2p\pi),\quad 1\leq i, i'\leq \dimV\; .
$$
\noindent\textbf{Step~2.} 
In this step, we prove~(\ref{eq:CenteredZ}), assuming that for each $j\geq1$, $\{Y_{j,k},\;k=0,\dots,n_j-1\}$ is $m$-dependent.
We shall apply Proposition~\ref{prop:cltMdep} under this additional assumption. We thus need to
show that~(\ref{eq:CLTMdep-LimDist})--(\ref{eq:CLTMdep-LimVarLJ}) hold. Relations~(\ref{eq:varianceZlim}),~(\ref{eq:CLTZaloneProc}) and
the continuous mapping theorem imply~(\ref{eq:CLTMdep-LimDist}) with 
$$
\bY_k=
\left[
\begin{array}{c}
Z_{1,\infty,k}^2-\PE[Z_{1,\infty,k}^2] \\
\vdots\\
Z_{\dimV,\infty,k}^2-\PE[Z_{\dimV,\infty,k}^2]
\end{array}
\right] \;.
$$
Since $\bZ_\centerdot$ is Gaussian, we have, for all $k,k'\geq0$ and all $i,i'=1,\dots,\dimV$,
\begin{align*}
\PCov\left(\bY_{k,i},\bY_{k',i'}\right)&=2\PCov^2\left(Z_{i,\infty,k},Z_{i',\infty,k'}\right)\\
&
=2\;
\limcons_{i,i'}^2\;
\left(\int_{-\infty}^{\infty}\vfoncsym_{i,i'}(\lambda)\;\rme^{\rmi \lambda(k'-k)}\,\rmd\lambda\right)^2\;.
\end{align*}
Hence Relation~(\ref{eq:LimcovSquare-neq1}) in Corollary~\ref{cor:Limcov} and the previous display
yield~(\ref{eq:CLTMdep-LimVar}). The final condition~(\ref{eq:CLTMdep-LimVarLJ}) follows from Relation~(\ref{eq:Limcov}) of
Corollary~\ref{cor:Limcov} with a covariance matrix $\Gamma$ with entries~(\ref{eq:GammaDef}). Applying
Proposition~\ref{prop:cltMdep} then yields~(\ref{eq:CenteredZ}), with $\Gamma$ given by~(\ref{eq:GammaDef}).

\medskip
\noindent\textbf{Step~3.} 
Let $K(t)$ be a non-negative infinitely differentiable function defined on $t\in\Rset$ whose support is included in $[-1/2,1/2]$
and such that $K(0)=1$. We will denote by $\hat{K}$ its Fourier transform,
$$
\hat{K}(\xi)=\int_{-\infty}^\infty K(t) \rme^{-\rmi \xi t} \, \rmd t \; .
$$
Observe that, by the assumptions on $K$, $\hat{K}(\xi)$ decreases faster than any polynomial as $|\xi|\to\infty$. 
In particular $\hat{K}(\xi)$ is integrable on $\xi\in\Rset$ and, for all $t\in\Rset$, 
\begin{align}
\label{eq:FourierInvK}
 K(t) =\frac1{2\pi}\int_{-\infty}^\infty \hat{K}(\xi) \rme^{\rmi \xi t} \, \rmd\xi \; .
\end{align}

The function $K$ will be used to approximate the $\vseq^{(m)}_{i,j}$ sequence by a sequence whose dependence structure can be
controlled. We thus define, for any $i=1,\dots,\dimV$, $j\geq0$ and $m\geq1$,
$$
\vseq^{(m)}_{i,j}(t)= \vseq_{i,j}(t) \,K(t/(m\decim_j)), 
$$
which vanishes for all $|t|\geq m\decim_j/2$,
and we define $Z^{(m)}_{i,j,k}$ and ${\vfonc_{i,j}}^{(m)}(\lambda)$ accordingly. 

Let $m$ be a fixed integer. Then, for all $j\geq0$, 
$$
\{[Z^{(m)}_{1,j,k}\,\,\dots\,\,Z^{(m)}_{\dimV,j,k}]^T,\,k\in\Zset\}
$$ 
is an $m$-dependent sequence of vectors.
We shall now show that $\{{\vfonc_{i,j}}^{(m)},\,j\geq0\}$ satisfy conditions similar to~(\ref{eq:unfiBoundvfonc}) 
and~(\ref{eq:Limitvfonc}) and then apply Step~2. Using~(\ref{eq:FourierInvK}) and~(\ref{eq:vfoncdef}) in the equation
$$
{\vfonc_{i,j}}^{(m)}(\lambda)= (2\pi)^{-1/2} \, \sum_{t\in\Zset} \vseq_{i,j}(t) \,K(t/(m\decim_j)) \rme^{-\rmi \lambda t}  \;, 
$$
we get that
\begin{align}
\label{eq:vfondMdep}
{\vfonc_{i,j}}^{(m)}(\lambda)=\frac{m}{2\pi}\int_{-\infty}^\infty \hat{K}(m\xi)\vfonc_{i,j}(\lambda-\decim_j^{-1}\xi)\,\rmd\xi \; .
\end{align}
It follows from Condition~(\ref{eq:unfiBoundvfonc}) that there exists a
constant $C>0$ such that for all $j\geq0$ and $\lambda\in[0,\pi)$, 
\begin{equation}
  \label{eq:vfoncBoundhT}
  \left|\vfonc_{i,j}(\lambda)\right|\leq 
C\; \decim_j^{1/2} (1+|\decim_j\lambda-\decim_j\lambda_{i,j}|)^{-\delta} 
\; .   
\end{equation}
Using~(\ref{eq:symVfonc}) and the $(2\pi)$-periodicity of $\vfonc_{i,j}$, we can express~(\ref{eq:vfoncBoundhT}) using the
symmetric   $(2\pi\decim_j)$-periodic function 
$\poldecrease_{2\pi\decim_j,\decim_j\lambda_{i,j}}(\xi)$ defined in Lemma~\ref{lem:polynomDecrease} and equal to
$(1+|\xi-\decim_j\lambda_{i,j}|)^{-\delta}$ for $0\leq\xi\leq\pi\decim_j$. With $\xi=\decim_j\lambda$, one gets
\begin{equation*}
\left|\vfonc_{i,j}(\lambda)\right|\leq C\; \decim_j^{1/2} \; \poldecrease_{2\pi\decim_j,\decim_j\lambda_{i,j}}(\decim_j\lambda)
\; ,\quad j\geq0,\;\lambda\in\Rset\;. 
\end{equation*}
Let $g(t)=m|\hat{K}(mt)|$ and observe that $\|g\|_1=\|\hat{K}\|_1<\infty$ and
$$
g(t)\leq c_0\;m\;(m|t|)^{-\delta-1}\leq c_0|t|^{-\delta-1}\quad\text{for all $|t|\geq1$ and $m\geq1$}\;,
$$
where $c_0$ is a positive constant such that $|\hat{K}(u)|\leq c_0|u|^{-\delta-1}$ for $|u|\geq1$.
Applying these bounds to~(\ref{eq:vfondMdep}) gives
$$
\left|{\vfonc_{i,j}}^{(m)}(\lambda)\right|\leq C\decim_j^{1/2}
\int_{-\infty}^{\infty}g(\xi)\;
\poldecrease_{2\pi\decim_j,\decim_j\lambda_{i,j}}(\decim_j\lambda-\xi)\;\rmd\xi \;.
$$
Applying Lemma~\ref{lem:polynomDecrease} to this convolution, we get
$$
\left|{\vfonc_{i,j}}^{(m)}(\lambda)\right|
\leq C\;\decim_j^{1/2}\;(1+\decim_j|\lambda-\lambda_{i,j}|)^{-\delta} \;,
$$
for different constants $C$ depending neither on $m\geq1$, $j\geq0$ nor on $\lambda\in[0,\pi)$. One has therefore the
following version of~(\ref{eq:unfiBoundvfonc}) for ${\vfonc_{i,j}}^{(m)}$, uniform in $m\geq1$:
\begin{align}
\label{eq:unfiBoundvfoncMdep}
\sup_{j\geq0} \;\;\sup_{m\geq1} \;\;\sup_{\lambda\in[0,\pi)} \decim_j^{-1/2}|{\vfonc_{i,j}}^{(m)}(\lambda)|(1+\decim_j|\lambda-\lambda_{i,j}|)^{\delta} < \infty \; .
\end{align}
To get a version of~(\ref{eq:Limitvfonc}) for ${\vfonc_{i,j}}^{(m)}$, observe that, by~(\ref{eq:vfondMdep}), we have
$$
\decim_j^{-1/2}{\vfonc_{i,j}}^{(m)}(\decim_j^{-1}\lambda+\lambda_{i,j})=\frac{m}{2\pi}\int_{-\infty}^\infty
\hat{K}(m\xi)\left[\decim_j^{-1/2}\vfonc_{i,j}(\decim_j^{-1}(\lambda-\xi)+\lambda_{i,j})\right]\,\rmd\xi. 
$$
Condition~(\ref{eq:unfiBoundvfonc}) implies that the term in brackets is bounded independently of $\xi$ and $j$
and hence by~(\ref{eq:Limitvfonc}) and dominated convergence, one has
\begin{align}
\label{eq:LimitvfoncMdep}
\lim_{j\to\infty}\decim_j^{-1/2}{\vfonc_{i,j}}^{(m)}(\decim_j^{-1}\lambda+\lambda_{i,j}) = 
{\vfonc_{i,\infty}}^{(m)}(\lambda) \quad\text{for all}\quad\lambda\in\Rset\;,
\end{align}
where
\begin{align}
\label{eq:DefLimitvfoncMdep}
{\vfonc_{i,\infty}}^{(m)}(\lambda) \eqdef
 \frac{m}{2\pi}\int_{-\infty}^\infty\hat{K}(m\xi)\vfonc_{i,\infty}(\lambda-\xi)\,\rmd\xi \; .
\end{align}
Note that ${\vfonc_{i,\infty}}^{(m)}(\lambda)$ is an approximating sequence of ${\vfonc_{i,\infty}}(\lambda)$ in the sense
that,
since $\vfonc_{i,\infty}$ is bounded (by~(\ref{eq:LimitvfoncBound})) and  continuous (by hypothesis), and 
since $(2\pi)^{-1}\int_{-\infty}^\infty\hat{K}(u)\,du=K(0)=1$, for all $\lambda\in\Rset$, 
\begin{align}
\label{eq:vFoncMToVfonc}
{\vfonc_{i,\infty}}^{(m)}(\lambda)=\frac{1}{2\pi}\int_{-\infty}^\infty\hat{K}(u)\vfonc_{i,\infty}(\lambda-u/m)\,du
\to \vfonc_{i,\infty}(\lambda)\quad\text{as $m\to\infty$}\;.
\end{align}
Relations~(\ref{eq:unfiBoundvfoncMdep}) and~(\ref{eq:LimitvfoncMdep}) are the corresponding versions of
Conditions~(\ref{eq:unfiBoundvfonc}) and~(\ref{eq:Limitvfonc}) for ${\vfonc_{i,j}}^{(m)}$. 
and since, we are in the $m$-dependent case, we may apply the result proved in Step~2, and obtain, as $j\to\infty$, 
\begin{equation}
\label{eq:CenteredZmDep}
n_j^{-1/2}
\left[
\begin{array}{c}
\sum_{k=0}^{n_j-1} \{Z_{1,j,k}^{(m)\,2} -\PE[Z_{1,j,k}^{(m)\,2}]\} \\
\vdots\\
\sum_{k=0}^{n_j-1} \{Z_{\dimV,j,k}^{(m)\,2} -\PE[Z_{\dimV,j,k}^{(m)\,2}]\}
\end{array}
\right]\cl \calN(0,\Gamma^{(m)})\;,
\end{equation}
where $\Gamma^{(m)}$ is the covariance matrix with entries
\begin{equation}
  \label{eq:GammaMDef}
\Gamma^{(m)}_{i,i'}\eqdef
2\pi\,\limcons_{i,i'}\, \int_{-\pi}^\pi \left|\sum_{p\in\Zset}
  {\vfoncsym_{i,i'}}^{(m)}(\lambda+2p\pi)\right|^2 \, \rmd\lambda ,
\quad 1\leq i, i'\leq \dimV\; ,
\end{equation}
where ${\vfoncsym_{i,i'}}^{(m)}$ is the equivalent of $\vfoncsym_{i,i'}$ in~(\ref{eq:vfoncsymDef}),
\begin{equation*}
  {\vfoncsym_{i,i'}}^{(m)}(\lambda)=
\frac12\left[\overline{{\vfonc_{i,\infty}}^{(m)}(-\lambda)}{\vfonc_{i',\infty}}^{(m)}(-\lambda)
+{\vfonc_{i,\infty}}^{(m)}(\lambda)\overline{{\vfonc_{i',\infty}}^{(m)}(\lambda)}\right],\quad\lambda\in\Rset\;.
\end{equation*}

To obtain the corresponding result~(\ref{eq:CenteredZ}) for the $\{Z_{i,j,k}\}$ sequence, we apply
\cite[Theorem~3.2]{billingsley:1999} as follows. We show that 
\begin{equation}
\label{eq:ApproxLimitCLT}
\lim_{m\to\infty} \Gamma^{(m)}_{i,i'}= \Gamma_{i,i'}\quad 1\leq i, i'\leq \dimV
\end{equation}
and, for all $i=1,\dots,\dimV$,
\begin{equation}
\label{eq:ApproxLimitCLTinterm}
\lim_{m\to\infty}\lim_{j\to\infty} 
\PVar\left(n_j^{-1/2}\sum_{k=0}^{n_j-1} Z_{i,j,k}^{(m)\,2}- 
n_j^{-1/2}\sum_{k=0}^{n_j-1} Z_{i,j,k}^{2}  \right) = 0 \; .
\end{equation}
Relation~(\ref{eq:ApproxLimitCLT}) says that the RHS of~(\ref{eq:CenteredZmDep}) converges in distribution to the RHS
of~(\ref{eq:CenteredZ}) and Relation~~(\ref{eq:ApproxLimitCLTinterm}) says that the LHS of~(\ref{eq:CenteredZmDep}) is a good approximation
to the LHS~(\ref{eq:CenteredZ}) by choosing $m$ arbitrary large.

To prove~(\ref{eq:ApproxLimitCLTinterm}), it is sufficient to establish the following equalities for all $i=1,\dots,\dimV$.
\begin{multline*}
\lim_{m\to\infty}\lim_{j\to\infty}\PCov\left(n_j^{-1/2}\sum_{k=0}^{n_j-1} Z_{i,j,k}^{(m)\,2}, n_j^{-1/2}\sum_{k=0}^{n_j-1}
  Z_{i,j,k}^{2} \right)\\
=\lim_{m\to\infty}\lim_{j\to\infty} \PVar\left(n_j^{-1/2}\sum_{k=0}^{n_j-1} Z_{i,j,k}^{(m)\,2} \right)
= \lim_{j\to\infty} \PVar\left(n_j^{-1/2}\sum_{k=0}^{n_j-1} Z_{i,j,k}^{2}\right)\; .
\end{multline*}
Using Relation~(\ref{eq:Limcov}) of Corollary~\ref{cor:Limcov}, the limits as $j\to\infty$ (an hence $n_j\to\infty$) in the
previous display are, respectively, 
$$
 \Gamma_{i,i}^{(m,\infty)}\eqdef
2\pi\,\limcons_{i,i'}\,\int_{-\pi}^\pi
\left|\sum_{p\in\Zset} {\vfoncsym_{i,i'}}^{(m,\infty)}(\lambda+2p\pi)\right|^2 \,
\rmd\lambda,\quad \Gamma_{i,i}^{(m)}
\quad\text{and}\quad \Gamma_{i,i} \; ,
$$
where $\Gamma_{i,i}^{(m)}$ is defined in~(\ref{eq:GammaMDef}) and $\Gamma_{i,i}$ in~(\ref{eq:GammaDef}) and
$$
{\vfoncsym_{i,i'}}^{(m,\infty)}(\lambda)=
\frac12\left[\overline{{\vfonc_{i,\infty}}^{(m)}(-\lambda)}{\vfonc_{i',\infty}}(-\lambda)
+{\vfonc_{i,\infty}}^{(m)}(\lambda)\overline{{\vfonc_{i',\infty}}(\lambda)}\right],\quad\lambda\in\Rset\;.
$$
Hence to prove~(\ref{eq:ApproxLimitCLT}) and~(\ref{eq:ApproxLimitCLTinterm}), it is sufficient to show that
\begin{equation}
  \label{eq:Cov-ConvLimInm}
  \lim_{m\to\infty}  \Gamma_{i,i'}^{(m,\infty)}=\lim_{m\to\infty}  \Gamma_{i,i'}^{(m)}=\Gamma_{i,i'},\quad i,i'=1,\dots,\dimV \;.  
\end{equation}
Observe first that Relations~(\ref{eq:unfiBoundvfoncMdep}) and~(\ref{eq:LimitvfoncMdep}) imply
\begin{align}
\label{eq:LimitvfoncMBound}
  \sup_{\lambda\in\Rset} \;\;\sup_{m\geq1}  \left|{\vfonc_{i,\infty}}^{(m)}(\lambda)\right|\; (1+|\lambda|)^{\delta} < \infty \;,
\end{align}
which is the uniform version of~(\ref{eq:LimitvfoncBound}). Eq.~(\ref{eq:Cov-ConvLimInm}) now follows
bfrom~(\ref{eq:vFoncMToVfonc}),~(\ref{eq:unfiBoundvfoncMdep}) and~(\ref{eq:LimitvfoncMBound}), and dominated convergence.
\end{proof}

\section{Proof of Theorem~\ref{thm:CLTlinearLocale}}
\label{sec:proof-theor-cltLocal}
The following proposition is the key point for proving Theorem~\ref{thm:CLTlinearLocale} since it shows how 
Condition~(\ref{eq:unfiBoundvfonc}) in Theorem~\ref{thm:CLTlinear} can be recovered for an approximation of the sample 
mean square, when using the alternative Condition~(\ref{eq:unfiBoundvfoncLocale}).
Condition~(\ref{eq:BoundvfoncLocaleNonZero}) in Theorem~\ref{thm:CLTlinearLocale} can then be used to control the sharpness
of the approximation.
\begin{proposition}
  \label{prop:ApproxLocalCondition}
Let $\{Z_{1,j,k},\;i=1,2, j\geq0, k\in\Zset\}$ be an array of $1$--dimensional decimated linear processes as defined
by~(\ref{eq:Zdef}). Assume that $\{\vseq_{1,j}(t),\;j\geq0,\;t\in\Zset\}$ satisfies~(\ref{eq:unfiBoundvfoncLocale}) for
$\delta>1/2$, a sequence $(\lambda_{1,j})$ taking its values in $[0,\pi)$ and some $\varepsilon>0$. 
Then there exists an array $\{\widehat{\vseq_{1,j}}(t),\;j\geq0,\;t\in\Zset\}$ whose Fourier series coincide with those
of $\{\vseq_{1,j}(t),\;j\geq0,\;t\in\Zset\}$ in $\varepsilon$-neighborhoods of the frequencies $\{\lambda_{1,j},\;j\geq0\}$
and satisfying~(\ref{eq:unfiBoundvfonc}), that is, such that
\begin{align}
\label{eq:vfoncWithHatequalsInNeighb}
&\widehat{\vfonc_{1,j}}(\lambda)=\vfonc_{1,j}(\lambda)
\quad\text{for all $\lambda\in(-\pi,\pi)$ such that $|\lambda-\lambda_{1,j}|\leq\varepsilon$}\;, \\
\label{eq:unfiBoundvfoncWithHat}
&\sup_{j\geq0} \sup_{\lambda\in[0,\pi)} \decim_j^{-1/2}|\widehat{\vfonc_{1,j}}(\lambda)|(1+\decim_j|\lambda-\lambda_{i,j}|)^\delta < \infty \;, 
\end{align}
and the following approximation holds.
\begin{equation}
  \label{eq:Rjdef}
  n_j^{-1/2}\left(\sum_{k=0}^{n_j-1} \{Z_{1,j,k}^{2} -\PE[Z_{1,j,k}^{2}]
\right)=
n_j^{-1/2}\sum_{k=0}^{n_j-1} \{\widehat{Z}_{1,j,k}^{2} -\PE[\widehat{Z}_{1,j,k}^{2}]\}
+R_j\;,
\end{equation}
where
\begin{equation}\label{eq:defZhat}
\widehat{Z}_{1,j,k}=\sum_{t\in\Zset}\widehat{\vseq_{1,j}}(\decim_jk-t)\xi_t\;,
\end{equation}
and, for some positive constant $C$ not depending on $j$,
\begin{equation}\label{eq:RjBound}
\PE\left[\left|R_j\right|\right]\leq C\;
\left[n_j^{1/2}I_j+I_j^{1/2}\right]\;,
\end{equation}
where
\begin{equation}\label{eq:Ijdef}
I_j\eqdef\int_{0}^\pi\1(|\lambda-\lambda_{1,\infty}|>\varepsilon)\;\left|\vfonc_{1,j}(\lambda)\right|^2\;\rmd\lambda \;.
\end{equation}
\end{proposition}
\begin{proof}
Let $\localset=[-\lambda_{1,\infty}-\varepsilon,-\lambda_{1,\infty}+\varepsilon]\cup[\lambda_{1,\infty}-\varepsilon,\lambda_{1,\infty}+\varepsilon]$. 
We write
$$
\vfonc_{1,j}(\lambda)=\widehat{\vfonc_{1,j}}(\lambda)+\widetilde{\vfonc_{1,j}}(\lambda),\quad\lambda\in(-\pi,\pi)\;,
$$
where $\widehat{\vfonc_{1,j}}(\lambda)=\1_{\localset}(\lambda)\vfonc_{1,j}(\lambda)$ so
that~(\ref{eq:vfoncWithHatequalsInNeighb}) holds.
We define $\widehat{\vseq_{1,j}}$, $\widetilde{\vseq_{1,j}}$ accordingly, so that 
$\vseq_{1,j}(t)=\widehat{\vseq_{1,j}}(t)+\widetilde{\vseq_{1,j}}(t)$ and, since $\widehat{\vfonc_{1,j}}$ and $\widetilde{\vfonc_{1,j}}$
are in $L^2(-\pi,\pi)$, $\widehat{\vseq_{1,j}}$ and $\widetilde{\vseq_{1,j}}$ are in $l^2(\Zset)$. Hence
$Z_{1,j,k}=\widehat{Z}_{1,j,k} + \widetilde{Z}_{1,j,k}$ with $\widehat{Z}_{1,j,k}$ defined by~(\ref{eq:defZhat}) and
\begin{equation}\label{eq:defZtilde}
\widetilde{Z}_{1,j,k}=\sum_{t\in\Zset}\widetilde{\vseq}_{1,j}(\decim_jk-t)\xi_t.
\end{equation}
Moreover, by~(\ref{eq:unfiBoundvfoncLocale}) and the definition of $\widehat{\vfonc_{1,j}}$,
Condition~(\ref{eq:unfiBoundvfoncWithHat}) holds.

We now show that the remainder $R_j$ defined by~(\ref{eq:Rjdef}) satisfies~(\ref{eq:RjBound}).
Observe that $\widehat{Z}_{1,j,k}$ and $\widetilde{Z}_{1,j,k}$ are centered and, since  
$$
\PE[\widehat{Z}_{1,j,k}\widetilde{Z}_{1,j,k}]=\int_{-\pi}^\pi\widehat{\vfonc_{1,j}}(\lambda)\widetilde{\vfonc_{1,j}}(\lambda)\rmd\lambda=0\;,
$$  
uncorrelated. Thus we get
$\PE[Z_{1,j,k}^{2}]=\PE[\widehat{Z}_{1,j,k}^{2}]+\PE[\widetilde{Z}_{1,j,k}^{2}]$ and hence
the remainder $R_j$ defined by~(\ref{eq:Rjdef}) is 
\begin{equation}
  \label{eq:RPQ}
R_j=P_j+2Q_j\quad\text{with}
\end{equation}
\begin{equation}
  \label{eq:PjQjdef}
P_j=n_j^{-1/2}\sum_{k=0}^{n_j-1}\{\widetilde{Z}_{1,j,k}^{2}-\PE[\widetilde{Z}_{1,j,k}^{2}]\}
\quad\text{and}\quad
Q_j=n_j^{-1/2}\sum_{k=0}^{n_j-1}\widetilde{Z}_{1,j,k}\widehat{Z}_{1,j,k}\;.
\end{equation}
We have
\begin{equation}
   \label{eq:PjBound}
  \PE[|P_j|]\leq 2 n_j^{1/2}\PE[\widetilde{Z}_{1,j,0}^{2}] 
= 2 n_j^{1/2}\sum_{t\in\Zset}\widetilde{\vseq}_{1,j}^2(t)
= 4
n_j^{1/2}\;I_j
\;.  
\end{equation}
by the Parseval Theorem and the definitions of $\widetilde{\vfonc_{1,j}}$ and $I_j$.

Using that  $\widehat{Z}_{1,j,k}$ and $\widetilde{Z}_{1,j,k'}$ are centered and uncorrelated, we have, 
using a standard formula for cumulants of products, for all $k,k'\in\Zset$,
\begin{eqnarray*}
\PCov\left(\widetilde{Z}_{1,j,k}\widehat{Z}_{1,j,k},\widetilde{Z}_{1,j,k'}\widehat{Z}_{1,j,k'}\right) =
 &\PCov\left(\widetilde{Z}_{1,j,k},\widetilde{Z}_{1,j,k'}\right)\PCov\left(\widehat{Z}_{1,j,k},\widehat{Z}_{1,j,k'}\right)\\
& +\mathrm{cum}\left(\widetilde{Z}_{1,j,k},\widehat{Z}_{1,j,k},\widetilde{Z}_{1,j,k'},\widehat{Z}_{1,j,k'}\right) \; .
\end{eqnarray*}
Hence, $\PVar\left(Q_j\right)=A_j+B_j$ where
$$
A_j=n_j^{-1}\sum_{k=0}^{n_j-1}\sum_{k'=0}^{n_j-1}\PCov\left(\widetilde{Z}_{1,j,k},\widetilde{Z}_{1,j,k'}\right)\PCov\left(\widehat{Z}_{1,j,k},\widehat{Z}_{1,j,k'}\right)
$$
and
$$
B_j=n_j^{-1}
\sum_{k=0}^{n_j-1}\sum_{k'=0}^{n_j-1}\mathrm{cum}\left(\widetilde{Z}_{1,j,k},\widehat{Z}_{1,j,k},\widetilde{Z}_{1,j,k'},\widehat{Z}_{1,j,k'}\right)
\;. 
$$
Denote by $\widehat{f}_j$ and $\widetilde{f}_j$ the respective spectral densities of the weakly stationary processes
$\widehat{Z}_{1,j,\centerdot}$ and $\widetilde{Z}_{1,j,\centerdot}$. Replacing the covariances in the defintion of $A_j$
by their respective expressions as Fourier coefficients of the spectral density, \eg 
$\PCov\left(\widetilde{Z}_{1,j,k},\widetilde{Z}_{1,j,k'}\right)=\int_{-\pi}^\pi\rme^{\rmi(k-k')\lambda}\widetilde{f}_j(\lambda)\rmd\lambda$, we get 
$$
A_j=n_j^{-1}\int_{-\pi}^\pi\int_{-\pi}^\pi\widehat{f}_j(\lambda)\widetilde{f}_j(\lambda')\left|\sum_{k=0}^{n_j-1}\rme^{\rmi
    k(\lambda+\lambda')}\right|^2 \;\rmd\lambda\rmd\lambda'\;,
$$
which implies that
\begin{equation}
\label{eq:AjhattildeBound}
0\leq A_j\leq 2\pi \;
\sup_{\lambda\in(-\pi,\pi)}\widehat{f}_j(\lambda)\times\int_{-\pi}^\pi\widetilde{f}_j(\lambda')\rmd\lambda'\;,
\end{equation}
where, in the last inequality, we used that, for any $\lambda'$, $\int_{-\pi}^\pi\left|\sum_{k=0}^{n_j-1}\rme^{\rmi
    k(\lambda+\lambda')}\right|^2\rmd\lambda=2\pi n_j$.
Observe that, by definition of $\widehat{Z}_{1,j,k}$ in \eqref{eq:defZhat},
$$
\PCov\left(\widehat{Z}_{1,j,0},\widehat{Z}_{1,j,k}\right)=
(2\pi)\int_{-\pi}^\pi\left|\widehat{\vfonc_{1,j}}(\lambda)\right|^2\rme^{\rmi\decim_jk\lambda}\;\rmd\lambda
$$
Using Lemma~\ref{lem:finiteFolding} with the $(2\pi)$-periodic function
$g(\lambda)=\left|\widehat{\vfonc_{1,j}}(\lambda)\right|^2\rme^{\rmi\decim_jk\lambda}$, we get 
$$
\PCov\left(\widehat{Z}_{1,j,0},\widehat{Z}_{1,j,k}\right)=
(2\pi)\decim_j^{-1}\int_{-\pi}^\pi\left(\sum_{p=0}^{\decim_j-1}
\left|\widehat{\vfonc_{1,j}}(\decim_j^{-1}(\lambda+2p\pi))\right|^2\right)\;\rme^{\rmi k\lambda}\;\rmd\lambda \; .
$$
Hence we have
$$
\widehat{f}_j(\lambda)=(2\pi)^{-1}\decim_j^{-1}\sum_{p=0}^{\decim_j-1}\left|\widehat{\vfonc_{1,j}}(\decim_j^{-1}(\lambda+2p\pi))\right|^2\;.
$$
Using~(\ref{eq:unfiBoundvfoncWithHat}), since  
$|\decim_j^{-1}(\lambda+2p\pi)|< \pi$ for $p=0,\dots,\decim_j-1$ and $\lambda\in(-\decim_j\pi,-\decim_j\pi+2\pi)$, we get 
$$
\widehat{f}_j(\lambda)\leq
C\;\sum_{p=0}^{\decim_j-1}(1+\left||\lambda+2p\pi|-\decim_j\lambda_{1,j}\right|)^{-2\delta},\quad \lambda\in(-\decim_j\pi,-\decim_j\pi+2\pi)\;.
$$
Using~(\ref{eq:supsupgj}) in Lemma~\ref{lem:DomConvArgumenttt} and that $\widehat{f}_j$ is $(2\pi)$-periodic, we obtain
\begin{equation}
  \label{eq:Bounfwidehatfj}
\sup_{j\geq0}\sup_{\lambda\in(-\pi,\pi)}\widehat{f}_j(\lambda) <\infty \;.  
\end{equation}
Moreover, we have
\begin{equation}
  \label{eq:Bounftildehatfj}
\int_{-\pi}^\pi\widetilde{f}_j(\lambda')\rmd\lambda'=\PVar\left(\widetilde{Z}_{1,j,0}\right)=
\sum_{t\in\Zset}\widetilde{\vseq}_{1,j}^2(t)=\int_{-\pi}^\pi\left|\widetilde{\vfonc_{1,j}}(\lambda)\right|^2\;\rmd\lambda\;,
\end{equation}
by the Parseval Theorem. Hence by~(\ref{eq:AjhattildeBound}), there is a positive constant $C$ such that
\begin{equation}
  \label{eq:AjBOUND}
  \left|A_j\right|\leq C\,\int_{-\pi}^\pi\left|\widetilde{\vfonc_{1,j}}(\lambda)\right|^2\;\rmd\lambda\: .
\end{equation}

We now consider $B_j$. Using \ref{it:A1} and the definitions of $\widetilde{Z}_{1,j,k}$ and $\widehat{Z}_{1,j,k}$ in
\eqref{eq:defZhat} and~(\ref{eq:defZtilde}), we have 
$$
\mathrm{cum}\left(\widetilde{Z}_{1,j,k},\widehat{Z}_{1,j,k},\widetilde{Z}_{1,j,k'},\widehat{Z}_{1,j,k'}\right)
=\kappa_4\sum_{t\in\Zset}\widetilde{\vseq}_{1,j}(\decim_jk-t)\widehat{\vseq}_{1,j}(\decim_jk-t)\widetilde{\vseq}_{1,j}(\decim_jk'-t)\widehat{\vseq}_{1,j}(\decim_jk'-t)\; .
$$
Hence
\begin{align}
\nonumber
|B_j|&\leq \kappa_4\sum_{t,\tau\in\Zset}\left|\widetilde{\vseq}_{1,j}(t)\widehat{\vseq}_{1,j}(t)
\widetilde{\vseq}_{1,j}(t+\decim_j\tau)\widehat{\vseq}_{1,j}(t+\decim_j\tau)\right|\\
\nonumber
&\leq  \kappa_4\left(\sum_{t\in\Zset}\left|\widetilde{\vseq}_{1,j}(t)\widehat{\vseq}_{1,j}(t)\right|\right)^2\\
\nonumber
&\leq
\kappa_4\;\sum_{t\in\Zset}\left|\widetilde{\vseq}_{1,j}(t)\right|^2\times\sum_{t\in\Zset}\left|\widehat{\vseq}_{1,j}(t)\right|^2\\
\label{eq:BjtildeatBound}
&=\kappa_4\;\int_{-\pi}^\pi\left|\widetilde{\vfonc_{1,j}}(\lambda)\right|^2\;\rmd\lambda \times 
\int_{-\pi}^\pi\left|\widehat{\vfonc_{1,j}}(\lambda)\right|^2\;\rmd\lambda
\end{align}
By definition of $\widehat{\vfonc_{1,j}}$ and~(\ref{eq:Bounfwidehatfj}), we have
$$
\int_{-\pi}^\pi\left|\widehat{\vfonc_{1,j}}(\lambda)\right|^2\;\rmd\lambda=\PVar\left(\widehat{Z}_{1,j,0}\right)=
\int_{-\pi}^\pi\widehat{f}_j(\lambda)\;\rmd\lambda\leq C\;,
$$
and hence
\begin{equation}
  \label{eq:BjBOUND}
  B_j\leq C \int_{-\pi}^\pi\left|\widetilde{\vfonc_{1,j}}(\lambda)\right|^2\;\rmd\lambda
\end{equation}
where $C$ denotes a positive constant not depending on $j$.
The bounds~(\ref{eq:AjBOUND}) and~(\ref{eq:BjBOUND}) and the definition of $\widetilde{\vfonc_{1,j}}$ yield
\begin{align*}
\PE\left[Q_j^2\right]=  \PVar(Q_j)=A_j+B_j&\leq 
2C\int_{-\pi}^\pi\left|\widetilde{\vfonc_{1,j}}(\lambda)\right|^2\;\rmd\lambda\\
& \leq 
4C\int_{0}^\pi\1\left(|\lambda-\lambda_{1,\infty}|>\varepsilon\right)
\left|\vfonc_{1,j}(\lambda)\right|^2\;\rmd\lambda=4C\;I_j
\end{align*}
by the definitions of $\widetilde{\vfonc_{1,j}}$ and $I_j$. This, with~(\ref{eq:RPQ}),~(\ref{eq:PjBound})
and Jensen's inequality yields
(\ref{eq:RjBound}), which concludes the proof.
\end{proof}
\begin{proof}[Proof of Theorem~\ref{thm:CLTlinearLocale}]
The general case can easily be adapted from the case $\dimV=1$, which we assume here. 
We apply Proposition~\ref{prop:ApproxLocalCondition}.
It follows from~(\ref{eq:vfoncWithHatequalsInNeighb}) and~(\ref{eq:unfiBoundvfoncWithHat}) that the assumptions of
Theorem~\ref{thm:CLTlinear} are verified for $\widehat{Z}_{1,j,k}$ and we obtain 
$$
n_j^{-1/2}\left(\sum_{k=0}^{n_j-1} \{\widehat{Z}_{1,j,k}^{2} -\PE[\widehat{Z}_{1,j,k}^{2}]\}
\right)\cl \calN(0,\Gamma)\;.
$$
It follows from~(\ref{eq:BoundvfoncLocaleNonZero}), that $R_j\cp0$ as $j \to\infty$,. Hence~(\ref{eq:Rjdef}) yields the
CLT~(\ref{eq:CenteredZ}), which concludes the proof. 
\end{proof}

\section{Proof of Theorem~\ref{thm:appl-spectr-dens}}
\label{sec:proof-theor-appli}

We shall use the following lemmas.
\begin{lemma}\label{lem:BiasControl}
Assume~\ref{item:betaAssump} and~\ref{item:appl-spectr-dens}. For $\varepsilon>0$ small enough, if $2\beta-1>2$, then
\begin{equation}
  \label{eq:BiasControl}
 \left|\int_{-\varepsilon}^{\varepsilon} \decim_j\left|\hat{W}(\decim_j\lambda)\right|^2\,f(\lambda)
   \; \rmd\lambda - f(0)\right| = O\left(\decim_j^{-2}\right) \; .
\end{equation}
\end{lemma}
\begin{proof}
Using~(\ref{eq:fSmooth}),  for $\varepsilon>0$ small enough, the left-hand side of~(\ref{eq:BiasControl}) is at most 
\begin{equation}
  \label{eq:LHSbiascontrol}
  \left|f(0)\,\int_{-\varepsilon}^{\varepsilon} \decim_j\left|\hat{W}(\decim_j\lambda)\right|^2 \; \rmd\lambda - f(0)\right| 
+ C\left|\int_{-\varepsilon}^{\varepsilon} \decim_j\left|\hat{W}(\decim_j\lambda)\right|^2 \,\lambda^2\; \rmd\lambda \right| \;,
\end{equation}
where $C$ is a positive constant.
To evaluate the first integral in~(\ref{eq:LHSbiascontrol}) we write $\int_{-\varepsilon}^{\varepsilon}=\int_{-\infty}^\infty-
\int_{-\infty}^{-\varepsilon}-\int_{\varepsilon}^\infty$. Using~\ref{item:appl-spectr-dens}, one gets
$$
\int_{-\varepsilon}^{\varepsilon} 
\decim_j\left|\hat{W}(\decim_j\lambda)\right|^2  \;\rmd\lambda =1 +O\left(\decim_j^{1-2\beta} \right)\; .
$$
The second integral in~(\ref{eq:LHSbiascontrol})  is bounded by
$$
\int_{-\infty}^{\infty}\decim_j\left|\hat{W}(\decim_j\lambda)\right|^2\,
 \lambda^2\; \rmd\lambda \leq C' \decim_j^{-2} \; ,
$$
where $C'=\int_{-\infty}^{\infty}\left|\hat{W}(\lambda)\right|^2\ \lambda^2\; \rmd\lambda<\infty$
by~\ref{item:appl-spectr-dens}. 
\end{proof}
\begin{lemma}\label{lem:BijApproxBound}
Assume~\ref{item:appl-spectr-dens} and define $B_j$ as in~(\ref{eq:defBij}). Then there is a positive constant $C$, such
that, for all $j\geq0$ and $\lambda\in(-\pi,\pi)$,
\begin{align}
  \label{eq:BijApprox}
\left|\decim_j^{-1/2}B_{j}(\lambda) - \hat{W}(\decim_j\lambda)\right| & \leq C
\;\decim_j^{-\beta} \;,\\
  \label{eq:BijSquareApprox}
\left|\left|B_{j}(\lambda)\right|^2 - \decim_j\left|\hat{W}(\decim_j\lambda)\right|^2\right| & \leq
C\,\left[\decim_j^{1-\beta}\,\left|\hat{W}(\decim_j\lambda)\right|
+\decim_j^{1-2\beta}\right]\;,
\end{align}
and, for any positive $\epsilon$,
\begin{equation}
  \label{eq:BijBound}
\sup_{\lambda\in[0,\pi)}\1(|\lambda|>\epsilon)\;\left|B_{j}(\lambda) \right| =
O(\decim_j^{1/2-\beta}) \;.
\end{equation}
\end{lemma}
\begin{proof}
By~(\ref{eq:defBij}), we have $B_{j}(\lambda) =\decim_j^{1/2}\hat{W}(\decim_j\lambda)
+ R_{j}(\lambda)$, where
\begin{equation*}
R_{j}(\lambda) = \sum_{p\neq0}\decim_j^{1/2}\hat{W}(\decim_j(\lambda+2p\pi))\;.
\end{equation*}
Using~\ref{item:appl-spectr-dens}, since $\beta>1$, we have
\begin{equation}
  \label{eq:RijBound}
\sup_{\lambda\in(-\pi,\pi)}\left|R_{j}(\lambda)\right|\leq C\;\decim_j^{1/2}\sum_{p>0}(1+(2p-1)\decim_j\pi)^{-\beta}
=O(\decim_j^{1/2-\beta})\;.
\end{equation}
The bound~(\ref{eq:RijBound}) gives~(\ref{eq:BijApprox}), which yield~(\ref{eq:BijBound}) by
using~\ref{item:appl-spectr-dens}.
The bound~(\ref{eq:BijSquareApprox}) follows from~(\ref{eq:BijApprox}) and 
$$
\left||z_1|^2-|z_2|^2\right|\leq
2|z_2|\times|z_1-z_2|+|z_1-z_2|^2
$$
applied with $z_1=B_{j}(\lambda)$ and $z_2=\decim_j^{1/2}\hat{W}(\decim_j\lambda)$. 
\end{proof}

\begin{lemma}\label{lem:SpecDensEstConditions}
Suppose that the assumptions of Theorem~\ref{thm:appl-spectr-dens} hold. Then for some arbitrary small $\varepsilon>0$,
\begin{equation}
  \label{eq:SpecEstUnifLocalBound}
  \sup_{j\geq0}\sup_{|\lambda|\leq\varepsilon}\decim_j^{-1/2}\left|\vfonc_{1,j}(\lambda)\right|(1+\decim_j|\lambda|)^{\beta}<\infty\;,
\end{equation}
\begin{equation}
  \label{eq:SpecEstVfoncLim}
  \lim_{j\to\infty}\decim_j^{-1/2}\vfonc_{1,j}(\decim_j^{-1}\lambda)=\vfonc_{1,\infty}(\lambda)
\quad\text{for all}\quad\lambda\in\Rset\;,
\end{equation}
with
\begin{equation}
  \label{eq:vfoncLimEstSpec}
  \vfonc_{1,\infty}(\lambda) = \afonc(0)\hat{W}(\lambda),\quad\lambda\in\Rset\;,
\end{equation}
and, as $j\to\infty$,
\begin{equation}
\label{eq:BoundvfoncLocalBeta}
\int_{0}^\pi\1(|\lambda|>\varepsilon)\;
|\vfonc_{1,j}(\lambda)|^2\;\rmd\lambda  =O\left(\decim_j^{1-2\beta}\right)\;.
\end{equation}
\end{lemma}
\begin{proof}
We have, for all $u\in\Zset$,
\begin{align*}
  W(\decim_j^{-1}u)&=\frac1{2\pi}\int_{-\infty}^\infty \hat{W}(\xi)\;\rme^{\rmi \decim_j^{-1}u\xi}\;\rmd\xi\\
&=\frac1{2\pi}\int_{-\pi}^\pi \left[\sum_{p\in\Zset}\decim_j\hat{W}(\decim_j(\lambda+2p\pi))\right]
\;\rme^{\rmi u\lambda}\;\rmd\lambda\;,
\end{align*}
hence the term in brackets is the Fourier series of $\{W(\decim_j^{-1}u),\;u\in\Zset\}$ and thus
$$
\sum_{u\in\Zset}W(\decim_j^{-1}u)\rme^{-\rmi\lambda u}=\sum_{p\in\Zset}\decim_j\hat{W}(\decim_j(\lambda+2p\pi))\;,
$$
which is some times called the Poisson formula. Inserting this in~(\ref{eq:vfoncSpecDef}), we obtain
\begin{equation}
  \label{eq:vfoncSpecEst}
\vfonc_{1,j}(\lambda) =\afonc(\lambda)\; B_{j}(\lambda),\quad \lambda\in(-\pi,\pi)\;,
\end{equation}
where $B_{j}$ is a $(2\pi)$--periodic function defined by
\begin{equation}
  \label{eq:defBij}
  B_{j}(\lambda) = \sum_{p\in\Zset}\decim_j^{1/2}\hat{W}(\decim_j(\lambda+2p\pi)) \; .
\end{equation}

Applying~(\ref{eq:vfoncSpecEst}) and~(\ref{eq:BijApprox}) in Lemma~\ref{lem:BijApproxBound},~\ref{item:appl-spectr-dens} and
that $|\afonc(\lambda)|$ is bounded in a neighborhood of the origin by~\ref{item:betaAssump}, we get, for some
arbitrary small $\varepsilon>0$, if $|\lambda|\leq\varepsilon$,
$$
\decim_j^{-1/2}|\vfonc_{1,j}(\lambda)|\leq C|\hat{W}(\decim_j\lambda)|+O\left(\decim_j^{-\beta}\right)
\leq C\left(1+\decim_j|\lambda|\right)^{-\beta}+O\left(\decim_j^{-\beta}\right)\;,
$$
where $C$ and the $O$-term do not depend on $\lambda$, which implies~(\ref{eq:SpecEstUnifLocalBound}).
 
Applying~(\ref{eq:vfoncSpecEst}) and~(\ref{eq:BijApprox}) in Lemma~\ref{lem:BijApproxBound}, we have,  as $j\to\infty$,
$$
\decim_j^{-1/2}\vfonc_{1,j}(\decim_j^{-1}\lambda)=
\afonc(\decim_j^{-1}\lambda)\hat{W}(\lambda)+O(\decim_j^{-\beta})
\to\afonc(0)\hat{W}(\lambda),\quad\lambda\in\Rset\;,
$$
where the limit holds by~\ref{item:betaAssump}. This gives~(\ref{eq:SpecEstVfoncLim}).

Applying~(\ref{eq:vfoncSpecEst}),~(\ref{eq:BijBound}) in Lemma~\ref{lem:BijApproxBound} and
$\int_{-\pi}^\pi|\afonc(\lambda)|^2\rmd\lambda<\infty$, we obtain~(\ref{eq:BoundvfoncLocalBeta}). 
\end{proof}

\begin{proof}[Proof of Theorem~\ref{thm:appl-spectr-dens}]
By~(\ref{eq:Xdef}) and~(\ref{eq:ZestSpecDef}), we have
$$
Z_{1,j,k}=\sum_{t\in\Zset}\decim_j^{-1/2}\sum_{v\in\Zset}W(k - \decim_j^{-1}t - \decim_j^{-1} v )\aseq(u-t)\,\xi_t
=\sum_{t\in\Zset}\vseq_{1,j}(\decim_j k-t)\,\xi_t\;,
$$
where
$$
\vseq_{1,j}(s) = \decim_j^{-1/2}\sum_{v\in\Zset}W(\decim_j^{-1}(s-v))\,\aseq(v),\quad s\in\Zset \;.
$$
Thus $\{Z_{1,j,t},\;j\geq0,t\in\Zset\}$ is an array of one-dimensional decimated linear processes as
in Definition~\ref{def:decimatedlinearprocess}. Moreover, the Fourier series~(\ref{eq:vfoncdef}) of $\vseq_{1,j}(s)$ is
\begin{equation}
  \label{eq:vfoncSpecDef}
\vfonc_{1,j}(\lambda) =\decim_j^{-1/2}\,\afonc(\lambda)\;\sum_{u\in\Zset}W(\decim_j^{-1}(u))\rme^{-\rmi\lambda u},\quad \lambda\in(-\pi,\pi)\;. 
\end{equation}
We let $\dimV=1$ and $\lambda_{1,j}=\lambda_{1,\infty}=0$ for all $j\geq0$, which yields (\ref{eq:integerCondition}),~(\ref{eq:Limitfreq}),~(\ref{eq:CondFreqEqual0orDifferent})
and~(\ref{eq:CondFreqEqualOrDifferent}) in Condition~\ref{item:FourierConditions}.
In view of Lemma~\ref{lem:SpecDensEstConditions}, Relation~(\ref{eq:Limitvfonc}) in Condition~\ref{item:FourierConditions} holds, as well as
Relations~(\ref{eq:unfiBoundvfoncLocale}) and~(\ref{eq:BoundvfoncLocaleNonZero}) in Theorem~\ref{thm:CLTlinearLocale} (recall
that in that theorem, Relation~(\ref{eq:unfiBoundvfoncLocale}) replaces Relation~(\ref{eq:unfiBoundvfonc}) in
Condition~\ref{item:FourierConditions}). Hence we may apply  Theorem~\ref{thm:CLTlinearLocale}.

We are now in a position to show first~(\ref{eq:limitEsp}) then~(\ref{eq:nCLTCondEstSpec}). 
Applying~(\ref{eq:covZ0Z1}), we have
\begin{equation}
  \label{eq:VarZSpecEst}
  \PE\left[Z_{1,j,0}^2\right] = \PVar\left(Z_{1,j,0}\right)
=\int_{-\pi}^{\pi} \left|\vfonc_{1,j}(\lambda)\right|^2\,\rmd\lambda\;.
\end{equation}
Using~(\ref{eq:vfoncSpecEst}),~(\ref{eq:BoundvfoncLocalBeta}) and then~(\ref{eq:BijSquareApprox}), this gives for any
$\epsilon>0$ small enough
\begin{align*}
\PE\left[Z_{1,j,0}^2\right] & =\int_{-\epsilon}^{\epsilon} \left|\afonc(\lambda)\,B_j(\lambda)\right|^2\,\rmd\lambda
+O\left(\decim_j^{1-2\beta}\right)\\
&= \int_{-\epsilon}^{\epsilon} \decim_j\left|\afonc(\lambda)\hat{W}(\decim_j\lambda)\right|^2\,\rmd\lambda
+C\,\int_{-\epsilon}^{\epsilon}
\decim_j^{1-\beta}\,\left|\hat{W}(\decim_j\lambda)\right|\rmd\lambda
+O\left(\decim_j^{1-2\beta}\right)\;.
\end{align*}
In the last line, since $\left|\afonc(\lambda)\right|^2=f(\lambda)$, by Lemma~\ref{lem:BiasControl}, the first term is
$f(0)+O\left(\decim_j^{-2}\right)$ and, by a change of variable, the second term is less than
$C\decim_j^{-\beta}\|\hat{W}\|_1$. 
Hence~(\ref{eq:limitEsp}) follows since $\beta>2$.

The bound~(\ref{eq:BoundvfoncLocalBeta}) yields~(\ref{eq:BoundvfoncLocaleNonZero}) under the condition $\decim_j^{1-2\beta}=o(n_j^{-1/2})$. 
Since $n_j$ is given by~(\ref{eq:njDefEstSpec}), the assumption~(\ref{eq:njCondEstSpec}) implies that
condition. Since~\ref{it:A1} holds as well, we may apply Theorem~\ref{thm:CLTlinearLocale} with $j=n$ and
obtain~(\ref{eq:nCLTCondEstSpec}). 
\end{proof}

\bibliography{lrd}
\bibliographystyle{plain}

\end{document}